\documentclass[12pt]{article}
\usepackage{CJK}
\usepackage{geometry}                
\usepackage{graphicx}
\usepackage{amssymb}
\usepackage{amsmath}
\usepackage{color}
\usepackage[titletoc]{appendix}
\usepackage{tikz}
\usepackage[numbers]{natbib}
\usepackage{hyperref}
\usepackage{algorithm}
\usepackage{algorithmic}
\usetikzlibrary{positioning}
\usetikzlibrary{arrows}
\newtheorem{remark}{Remark}

\newtheorem{theorem}{Theorem}[section]
\newtheorem{corollary}{Corollary}[section]
\newtheorem{lemma}{Lemma}[section]

\newenvironment{proof}{{\noindent\it Proof}\quad}{\hfill $\square$\par} 
\numberwithin{figure}{section}
\numberwithin{equation}{section}
\newcommand{\relu}{\mbox{{\rm ReLU}}}

\begin{document}

\title{ReLU Deep Neural Networks \\and Linear Finite Elements}
\author{Juncai He
	\thanks{School of Mathematical Sciences, Peking University, Beijing 100871, China. Email: juncaihe@pku.edu.cn}
	\and
	Lin Li
	\thanks{Beijing International Center for Mathematical Research, Peking University, Beijing 100871, China. Email: lilin1993@pku.edu.cn}
	\and 
	Jinchao Xu
	\thanks{Department of Mathematics, Pennsylvania State University, State College, PA 16802, USA. Email: xu@math.psu.edu}
	\and Chunyue Zheng
	\thanks{Department of Mathematics, Pennsylvania State University, State College, PA 16802, USA. Email: cmz5199@psu.edu}
}
  \date{}                                        

\maketitle
\begin{abstract}
	In this paper, we investigate the relationship between deep neural
	networks (DNN) with rectified linear unit (ReLU) function as the
	activation function and continuous piecewise linear (CPWL)
	functions, especially CPWL functions from the simplicial linear
	finite element method (FEM). 
	We first consider the special case of FEM.  By exploring the DNN
	representation of its nodal basis functions, 
	we present a ReLU DNN representation of CPWL in FEM. We
	theoretically establish that at least $2$ hidden layers are needed
	in a ReLU DNN to represent any linear finite element functions in
	$\Omega \subseteq \mathbb{R}^d$ when $d\ge2$.  Consequently, for
	$d=2,3$ which are often encountered in scientific and engineering
	computing, the minimal number of two hidden layers are necessary and
	sufficient for any CPWL function to be represented by a ReLU
	DNN. Then we include a detailed account on how a general CPWL in
	$\mathbb R^d$ can be represented by a ReLU DNN with at most
	$\lceil\log_2(d+1)\rceil$ hidden layers and we also give an
	estimation of the number of neurons in DNN that are needed in such a
	representation.  Furthermore, using the relationship between DNN and
	FEM, we theoretically argue that a special class of DNN models with
	low bit-width are still expected to have an adequate representation
	power in applications.  Finally, as a proof of concept, we present
	some numerical results for using ReLU DNNs to solve a two point
	boundary problem to demonstrate the potential of applying DNN for
	numerical solution of partial differential equations.
\end{abstract}

\section{Introduction}
\label{sec:into}

In recent years, deep learning models have achieved unprecedented
success in various tasks of machine learning or artificial
intelligence, such as computer vision, natural language processing and
reinforcement learning \cite{lecun2015deep}.  One main technique in
deep learning is deep neural network.  A typical DNN model is based on
a hierarchy of composition of linear functions and a given nonlinear
activation function.  However, why DNN models can work so well is
still unclear.  

Mathematical analysis of DNN can be carried out using many different
approaches.  One approach is to study the approximation properties of the
function class provided by DNN.   The approximation property of DNN is
relevant to the so-called expressive power \cite{cohen2016expressive} of a DNN model. 
Early studies of approximation properties of DNN can be traced back in
\cite{hornik1989multilayer} and \cite{cybenko1989approximation}
where the authors established some approximation properties for the
function classes given by a feedforward neural network with a single
hidden layer.  
Further error estimates for such neural networks in terms of number of
neurons can be found in \cite{jones1992simple} for sinusoidal
activation functions and in \cite{barron1993universal} for more
general sigmoidal activation functions.  There are many other papers
on this topic during the 90s and a good review of relevant works can
be found in \cite{ellacott1994aspects} and
\cite{pinkus1999approximation}.

There are many different choices of activation functions.  In fact, as
shown in \cite{leshno1993multilayer}, a neural network with a single
hidden layer can approximate any continuous function for any
activation function which is not a polynomial.  Among all the
activation functions, the so-called rectified linear unit (ReLU)
activation function \cite{nair2010rectified}, namely
$\relu(x)=\max(x,0)$, has emerged to be one of the most popular
activation functions used in the deep learning literature and
applications.  \cite{shaham2016provable} presents an approximation of
ReLU DNNs by relating to wavelets.  Recently,
\cite{klusowski2018approximation} establish $L^{\infty}$ and $L^2$
error bounds for functions of many variables that are approximated by
linear combinations of ReLU. \cite{zhou2018universality} presents
rates of approximation by deep CNNs for functions in the Sobolev space
$H^r(\Omega)$ with $r > 2 + d/2$.  This paper is devoted to some further
mathematical analysis of DNN models with ReLU as the activation
function.

It is not difficult to see that the following statement is true:
``{\it
	Every ReLU DNN function in $\mathbb R^d$ represents a continuous
	piecewise linear (CPWL) function defined on a number of polyhedral
	subdomains.}''
One important recent development is that the converse of the above
statement has also been proven true.  More specifically, the following result is
established by \cite{arora2016understanding} based on an earlier result
by \cite{tarela1999region} on lattice representation of DNN:
``{\it  Every CPWL function in $\mathbb{R}^d$ can be represented by a ReLU
	DNN model with at most $\lceil\log_2(d+1)\rceil$ hidden layers.}''

Motivated by this result, we study the following two questions on the DNN
representation of a given CWPL function:
\begin{enumerate}
	\item How many numbers of neurons are needed?
	\item What is the minimal number of layers that are needed?
\end{enumerate}

To answer the first question, in this paper, we will go through the
proof of this representation result to give some explicit estimations
of the number of neurons that are needed in a DNN to represent a given
CPWL function.  As a result, we find that the number of neurons that
are needed for a DNN to represent a CPWL on $m$-subdomains can be as
large as $\mathcal O(d2^{mm!})$!

In order to obtain DNN representation with fewer numbers of neurons, in
this paper, we consider a special class of CPWL functions, namely the
linear finite element (LFE) functions \cite{ciarlet2002finite}
defined on a collection of special subdomains, namely simplexes in
$\mathbb R^d$.  As every finite element function can be written as a
linear combination of nodal basis functions, it suffices to study DNN
representation of any given nodal basis function.  To represent a
nodal basis function by a DNN, we do not need to consider the
complicated domain partition related with lattice
representation(\cite{tarela1999region}), which is important in
representing general piecewise linear functions in
$\mathbb{R}^d$(~\cite{arora2016understanding}).  We prove that a
linear finite element function with $N$ degrees of freedom can be
represented by a ReLU DNN with at most $\mathcal O(d\kappa^dN)$ number of
neurons with $\mathcal O(d)$ hidden layers where $\kappa\ge 2$)
depends on the shape regularity of the underlying finite element
grid. 

To answer the second question, we will again consider the linear
finite element functions.  In this paper, we will show (see Theorem~\ref{lowerbound}) that at least $2$ hidden layers are needed for a ReLU DNN to
represent any linear finite element function in a bounded domain $\Omega
\subset \mathbb{R}^d$ or $\mathbb{R}^d$ when $d\ge2$.  The
$\lceil\log_2(d+1)\rceil$ number of hidden layers is also optimal for
$d=2, 3$.  Whether this number is also optimal for $d>3$ is still an
open problem.

In real applications, many efforts have been made to compress the deep
neural networks by using heavily quantized weights,
c.f.~\cite{yin2018binaryrelax}.  Especially, binary and ternary
weight models not only give high model compression rate, but also
eliminate the need of most floating-point multiplications during
interface phase. In particular, for some small data sets such as MINST
\cite{lecun1998gradient} and CIFAR-10 \cite{krizhevsky2009learning},
the ternary CNNs \cite{li2016ternary} are shown to have the same
accuracy as that of the original CNN.  Using the special structure for
representing any CPWL functions by ReLU DNNs, we provide certain
theoretical justification of the use of ternary CNNs.  Furthermore, we
also present a modified version of those models with some rigorous
mathematical justifications. 

Another topic that will be investigated in the paper is the
application of artificial neural networks for differential equations.
This topic can be traced back to \cite{meade1994numerical,
	meade1994solution, gobovic1994analog} in which collocation methods
are studied.  Recently, there are increased new research interests in the
literature for the application of deep neural networks for numerical
approximation of nonlinear and high dimensional PDEs as in
\cite{han2017overcoming, weinan2017deep, khoo2017solving}.  Based on
our result about the relationship between FEM and ReLU DNNs, we
discuss the application of ReLU DNN for solving PDEs with
respect to the convergence properties. In particular, we use an 1D
example to demonstrate that a Galerkin method using ReLU DNN 
can lead to better approximation result than adaptive finite element method
that has exactly the same number of degrees of freedom as in the ReLU
DNN. 

The remaining sections are organized as follows. In \S\ref{sec:reludnn_CPWL}, we introduce some notation and preliminary
results about ReLU DNNs. In \S\ref{sec:FEM_DNN} we investigate
the relationship between FEM and ReLU DNN. In \S\ref{sec:FEM_DNN1} we prove that at least 2 hidden layers are needed
to represent any LFE functions by ReLU DNN in $\Omega \subset
\mathbb{R}^d$ for $d\ge 2$.  In \S\ref{sec:cpwl}, we give the
self contained proof of representing CPWL and LFE functions with
$\lceil\log_2(d+1)\rceil$ hidden layers and give the size estimation.
In \S\ref{sec:specialstruc} we show that a special structure of
ReLU DNN can also recover any CPWL function.  In \S\ref{sec:appliPDE} we investigate the application of DNN for numerical
PDEs. In \S\ref{sec:conclusion} we give concluding remarks.

\section{Deep neural network (DNN) generated by ReLU}
\label{sec:reludnn_CPWL}
In this section, we briefly discuss the definition and properties of
the deep neural networks generated by using ReLU as the activation
function.

\subsection{General DNN}
Given $n, m\ge 1$, the first ingredient in defining a deep neural
network (DNN) is (vector) linear functions of the form
\begin{equation}\label{thetamap}
	\Theta:\mathbb{R}^{n}\to\mathbb{R}^{m} ,
\end{equation}as $\Theta(x)=Wx+b$ where
$W=(w_{ij})\in\mathbb{R}^{m\times n}$, $b\in\mathbb{R}^{m}$.

The second main ingredient is a nonlinear activation function, usually
denoted as 
\begin{equation}\label{sigma}
	\sigma: \mathbb{R} \to \mathbb{R}.
\end{equation} 
By applying the function to each component, we can extend this
naturally to 
$$
\sigma:\mathbb R^{n}\mapsto \mathbb R^{n}.
$$
Given $d,c, k\in\mathbb{N}^+$ and  
$$
n_1,\dots,n_{k}\in\mathbb{N} \mbox{ with }n_0=d, n_{k+1}=c, 
$$
a general DNN from $\mathbb{R}^d$ to $\mathbb{R}^c$ is given by
\begin{align*}
	f(x) &= f^k(x), \\
	f^{\ell}(x) &= [ \Theta^{\ell} \circ \sigma](f^{\ell-1}(x)) \quad \ell = 1:k,
\end{align*}
with $f^{0}(x) = \Theta (x)$. 
The following more concise notation is often used in computer science literature:
\begin{equation}
	\label{compress-dnn}
	f(x) = \Theta^{k}\circ \sigma \circ \Theta^{k-1} \circ \sigma \cdots \circ \Theta^1 \circ \sigma \circ \Theta^0(x),
\end{equation}
here $\Theta^i: \mathbb{R}^{n_{i}}\to\mathbb{R}^{n_{i+1}}$ are linear
functions as defined in \eqref{thetamap}.  Such a DNN is called a $(k+1)$-layer DNN, and is said
to have $k$ hidden layers. Unless otherwise stated, all layers mean hidden layers 
in the rest of this paper. The size of this DNN is
$n_1+\cdots+n_k$.  In this paper, we mainly consider a special activation function, known as the {\it rectified linear unit}
(ReLU), and defined as $\relu: \mathbb R\mapsto \mathbb R$,
\begin{equation}
	\label{relu}
	\relu(x)=\max(0,x), \quad x\in\mathbb{R}. 
\end{equation}
A ReLU DNN with $k$ hidden layers might be written as:
\begin{equation}
	\label{relu-dnn}
	f(x) = \Theta^{k}\circ \relu \circ \Theta^{k-1} \circ \relu \cdots \circ \Theta^1 \circ \relu \circ \Theta^0(x).
\end{equation}

We note that $\relu$ is a continuous piecewise linear (CPWL) function.
Since the composition of two CPWL functions is obviously still a CPWL
function, we have the following simple observation(\cite{arora2016understanding}).
\begin{lemma}\label{dnn-cpwl}
	Every ReLU DNN: $\mathbb{R}^d\to\mathbb{R}^c$ is a continuous
	piecewise linear function.  More specifically, given any ReLU DNN,
	there is a polyhedral decomposition of $\mathbb R^d$ such that this
	ReLU DNN is linear on each polyhedron in such a decomposition.
\end{lemma}

Here is a simple example for the ``grid" created by some 2-layer ReLU DNNs in $\mathbb{R}^2$.

\begin{figure}[ht]
	\includegraphics[width=.3\textwidth]{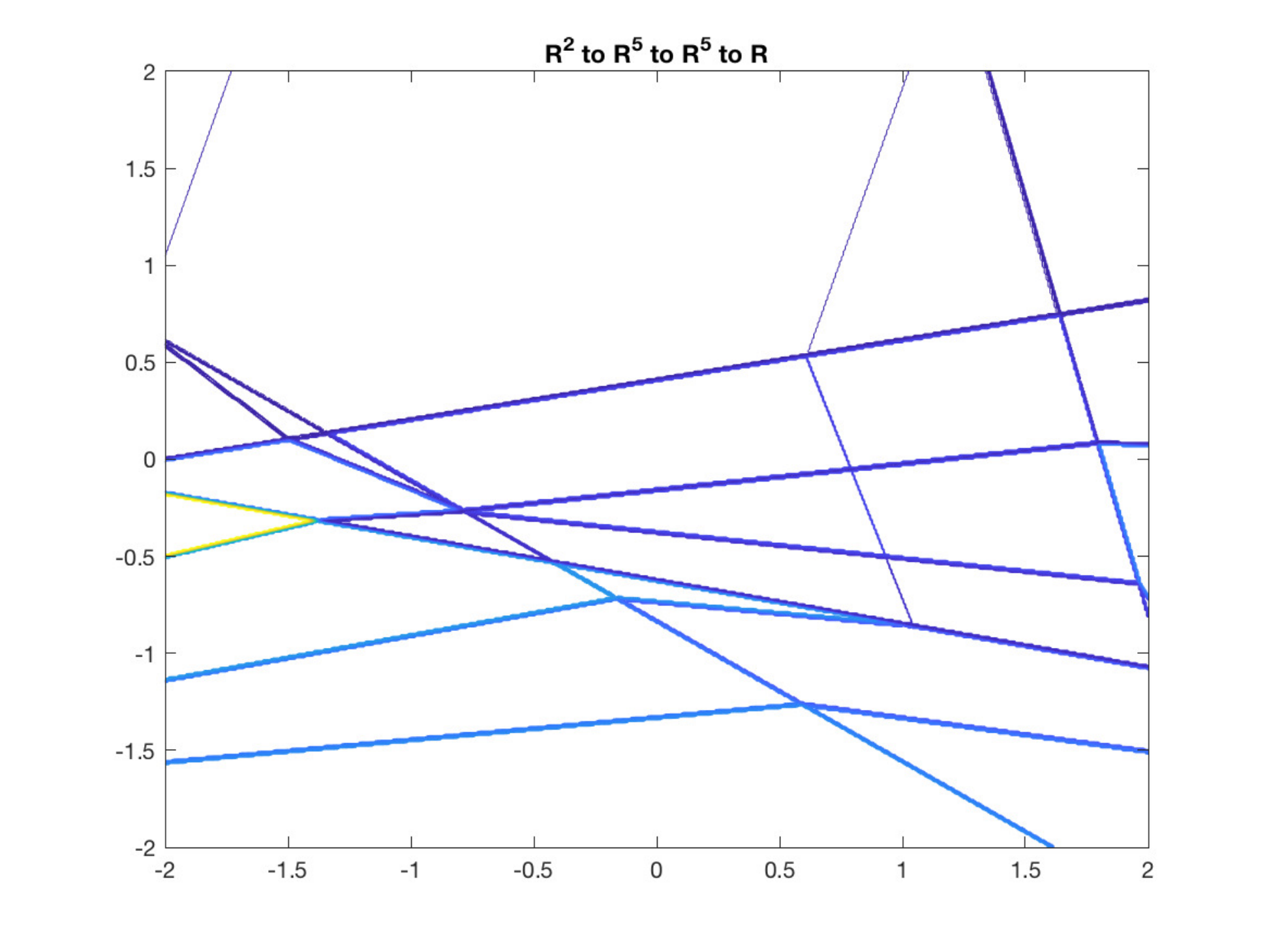}  
	\includegraphics[width=.3\textwidth]{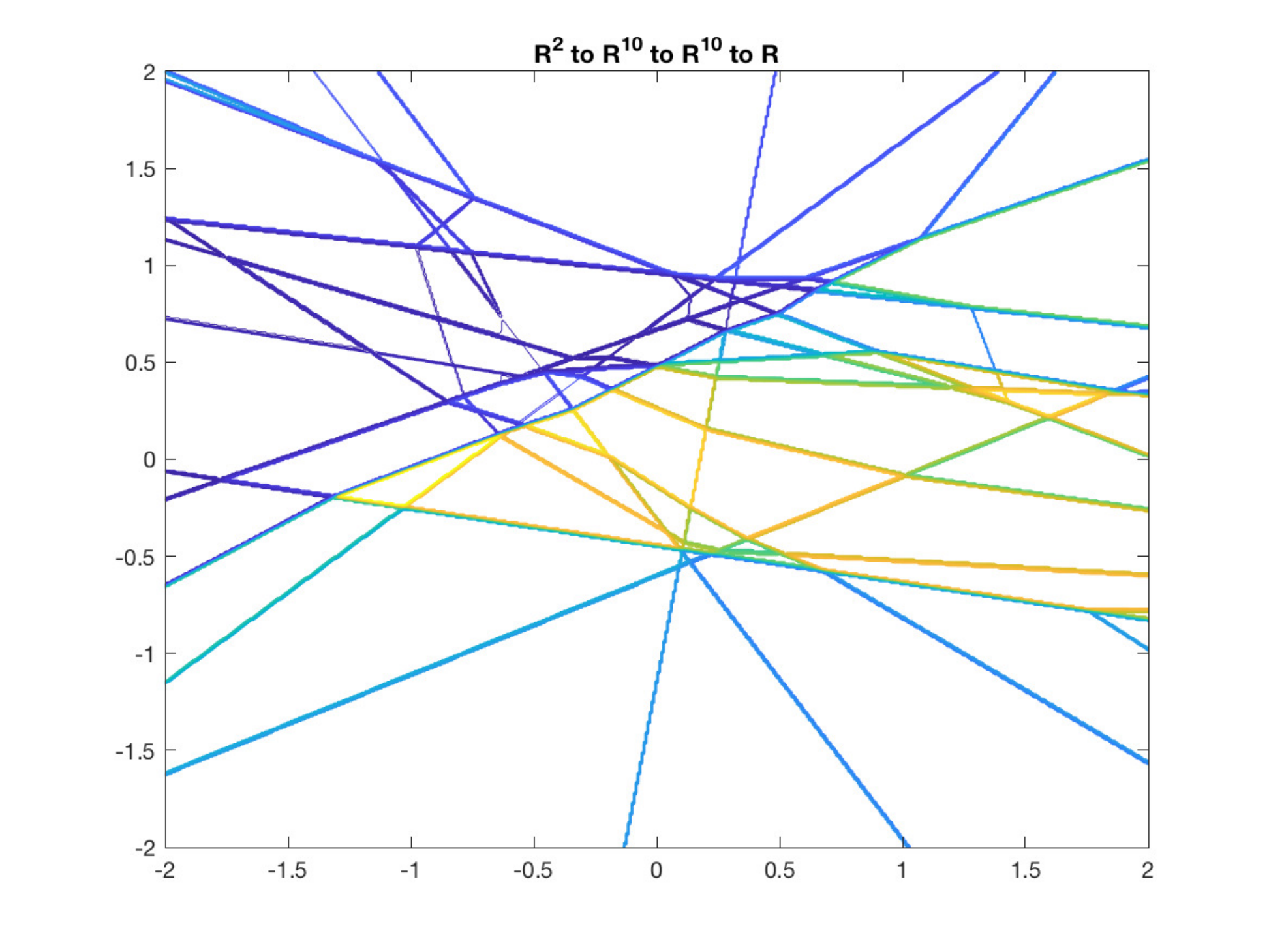}  
	\includegraphics[width=.3\textwidth]{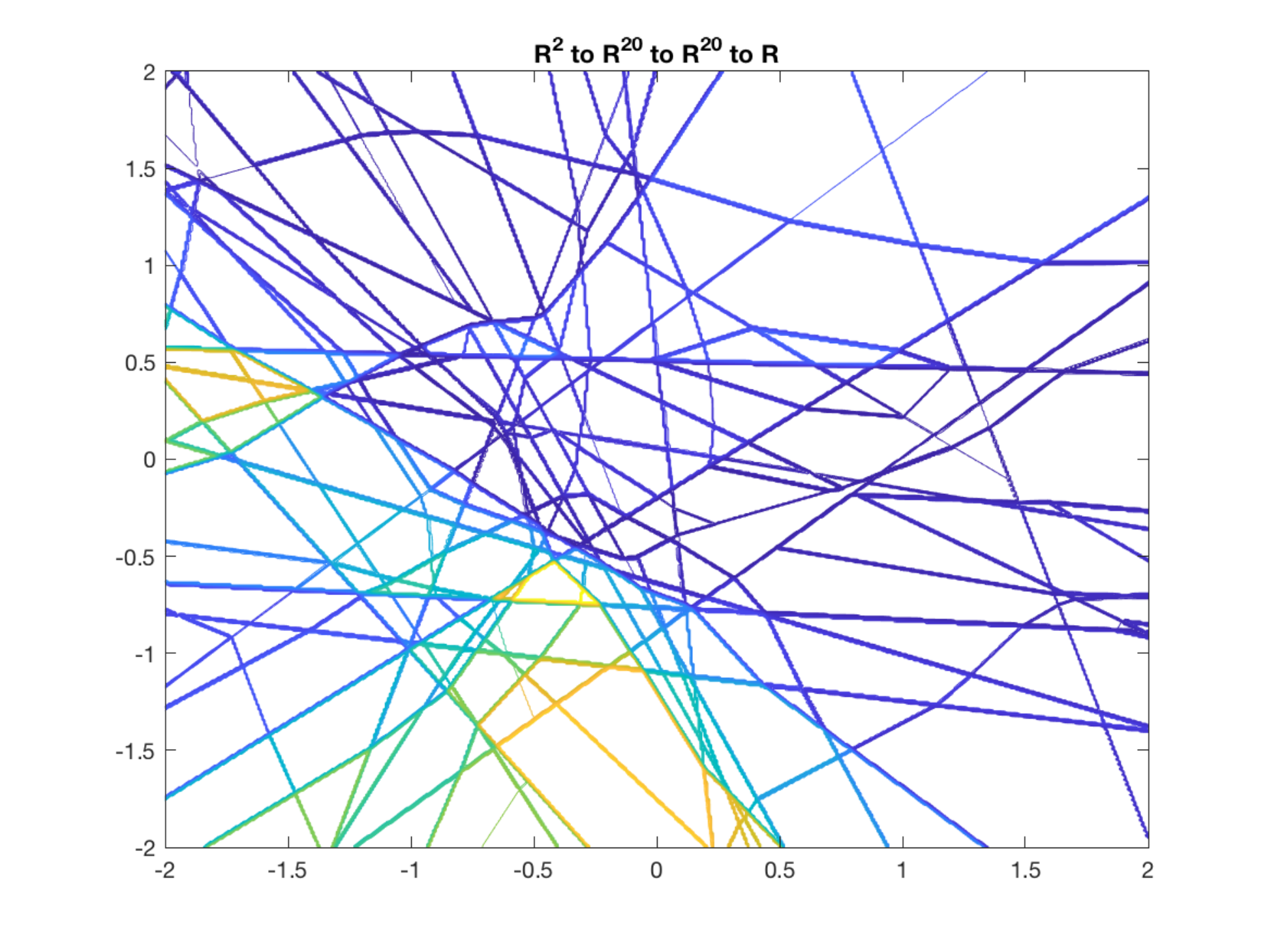}  
	\caption{Projections of the domain partitions formed by 2-layer ReLU DNNs with sizes $(n_0, n_1, n_2, n_3)= (2, 5, 5, 1), (2, 10, 10, 1) \text{and} (2, 20, 20, 1)$ with random parameters.}
	\label{fig:dnn-region}
\end{figure}

In the rest of the paper, we will use the terminology of CPWL to
define the class of functions that are globally continuous and locally
linear on each polyhedron in a given finite polyhedral decomposition of
$\mathbb R^d$. 

For convenience of exposition,  we introduce the following notation:
\begin{equation}
	\begin{aligned}
		{\rm{DNN}_J} :=\{& f:f=
		\Theta^J \circ \relu\circ \Theta^{J-1} \cdots \relu\circ \Theta^0(x), \\
		&\Theta^\ell \in \mathbb{R}^{n^{\ell+1} \times (n^\ell+1)}, \quad n^0 = d, \quad n^{J+1} = 1, \quad n^\ell \in \mathbb{N}^+\}.
	\end{aligned}
\end{equation}
Namely ${\rm{DNN}_J}$ represents the DNN model with $J$ hidden layers and
ReLU activation function with arbitrary size. 

%

\subsection{A shallow neural network ${\rm DNN}_1$} 
We note that for $J=0$, ${\rm DNN}_0$ is a simple function space of
global linear functions, which is often used in classic statistical
analysis such as linear regression.  The structure of ${\rm DNN}_J$
gets more interesting as $J$ becomes larger. We shall now
discuss the simple case when $J=1$, namely  
\begin{equation}
	\label{DNN1}
	{\rm DNN}_1^m=
	\left\{
	f: f=\sum_{i=1}^m \alpha_i\relu(w_i x+b_i)+\beta
	\right\},
\end{equation}
where $\alpha_i, b_i, \beta \in \mathbb{R}, w_i\in \mathbb{R}^{1\times d}$, for $i=1,2,...,m$.
Here we introduce the superscript $m$ to denote the number of
neurons. This simple neural network already has rich mathematical
structures and approximation properties. 
Given a bounded domain $\Omega\subset \mathbb R^d$, we introduce the
following notation
\begin{equation}
	\label{DNN1Omega}
	{\rm DNN}_1^m(\Omega)=
	\left\{
	f: f(x) \in {\rm DNN}_1^m, \quad  x\in \Omega \subset \mathbb R^d
	\right\},
\end{equation}
as the restriction of ${\rm DNN}_1^m$ on $\Omega$.

Approximation property for the function class ${\rm DNN}_1^m(\Omega)$
has been much studied in the literature.  For example, in
\cite{hornik1989multilayer} and \cite{cybenko1989approximation},
${\rm DNN}_{1}^m(\Omega)$ is proved to be dense in $C^0(\Omega)$ as $m
\to \infty$, which is known as universal approximation. There are also many works devoted to the asymptotic
error estimates.  For example,  \cite{barron1993universal}
established the following estimate:
\begin{equation}
	\label{Barron0}
	\inf_{g \in {\rm DNN}_{1}^m(\Omega)} \|f - g\|_{0, 2, \Omega} \lesssim
	|\Omega|^{1/2} m^{-{1\over2}}
	\int_{\mathbb R^d}|\omega|_\Omega|\hat f(\omega)|d\omega,
\end{equation}
where $|\Omega|$ denotes the volume of $\Omega$ and 
$$
|\omega|_\Omega=\sup_{x\in\Omega}|\omega\cdot (x-x_{\Omega})|,
$$
for some point $x_\Omega\in \Omega$.

For a given set of $w^i$ and $b^i$, it is tempting to think the
functions in ${\rm DNN}^m_1$ are generated by $\{\relu(w_i
x+b_i)\}_{i=1}^m$.  In such a consideration, the following result is of great theoretical interest. 
The proof will be seen in Section~\ref{sec:FEM_DNN1}.
\begin{theorem}\label{linearindep}
	$\{\relu(w_ix+b_i)\}_{i=1}^m$ are linearly independent if $(w_i,
	b_i)$ and $(w_j, b_j)$ are linearly independent in
	$\mathbb{R}^{1\times (d+1)} $ for any $i \neq j$.
\end{theorem}

In real applications, $w_i$ and $b_i$ are variables. 
As a result, 
${\rm  DNN}^m_1$ 
is generated by variable basis functions
$\{\relu(w_i x+b_i)\}_{i=1}^m$ and in particular   ${\rm  DNN}^m_1$
is a nonlinear space which is expected to have certain nonlinear
approximation property as discussed in 
\cite{devore1998nonlinear}.

\section{Linear finite element (LFE) function as a DNN}\label{sec:FEM_DNN}
In this section, we consider a special CPWL function space, namely the
space of linear simplicial finite element functions. We will first give a brief
description of finite element method and give a constructive
proof that any linear simplicial finite element function can be
represented by a ReLU DNN.  

\subsection{Linear finite element spaces}\label{sec:FEM}
The finite element method (FEM), as a popular numerical method for
approximating the solutions of partial differential equations (PDEs),
is a well-studied subject (\cite{ciarlet2002finite},\cite{brenner2007mathematical}).  The finite element function space is
usually a subspace of the solution space, for example, the space of piecewise
linear functions over a given mesh. In \cite{arora2016understanding}, it is shown that piecewise linear functions can be written as ReLU
DNNs, which will be discussed in details later. By exploring the relationship between FEM and ReLU DNN, we
hope to shed some new light on how DNN works in this special case.

Assuming that $\Omega \subset \mathbb{R}^d$ is a bounded domain. We consider a
special finite element function class consisting of CPWL functions
with respect to a simplicial partition of $\Omega$.  Such simplicial
partitions are often known as finite element grids or meshes.  Some
typical finite element grids are shown in
Figure~\ref{typicalfemgrid123} for $d=1, 2, 3$.
\begin{figure}[htbp]
	\centering
	\begin{minipage}[t]{0.3\textwidth}
		\centering
		\includegraphics[width=5cm]{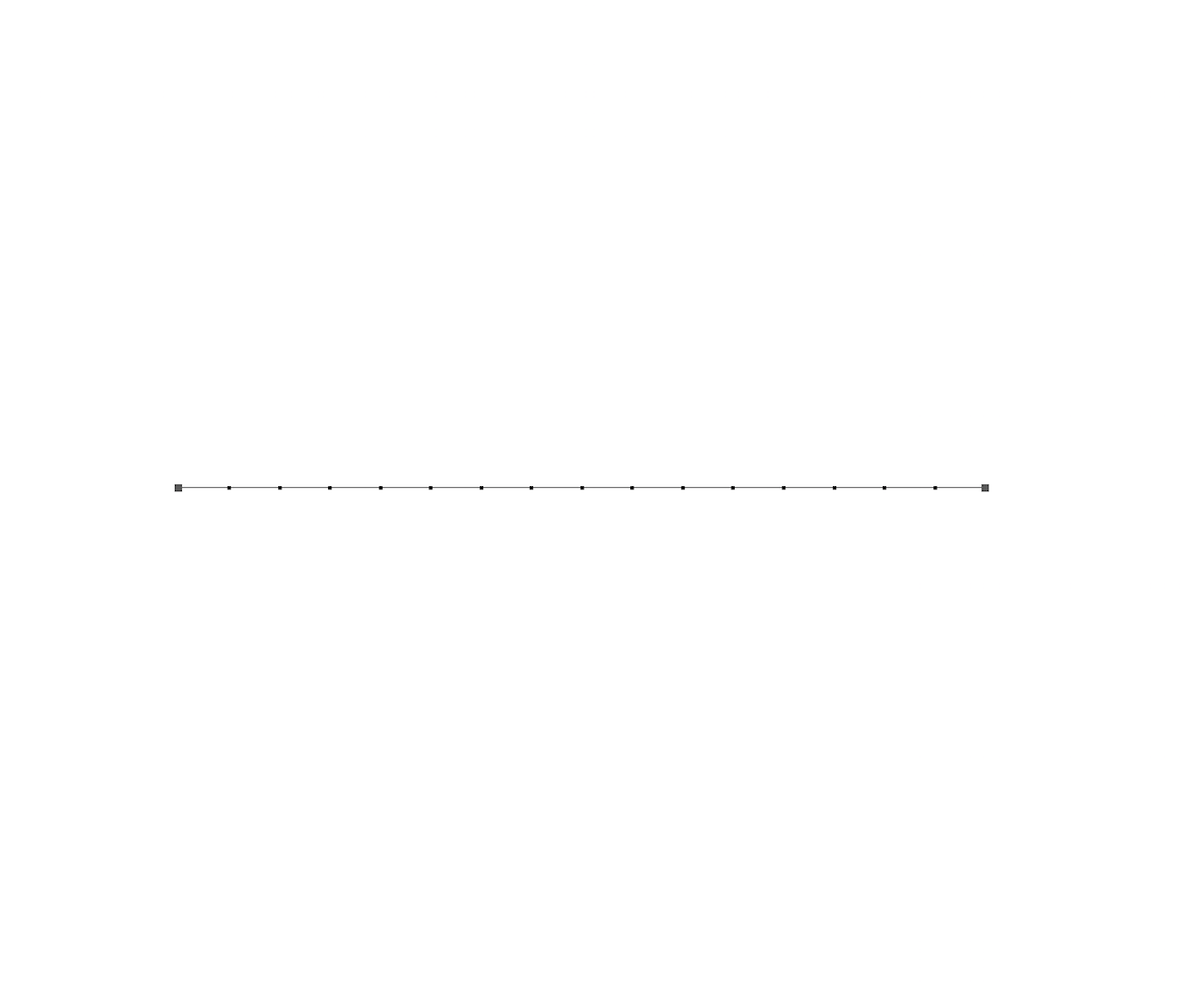}
	\end{minipage}
	\begin{minipage}[t]{0.3\textwidth}
		\centering
		\includegraphics[width=4cm]{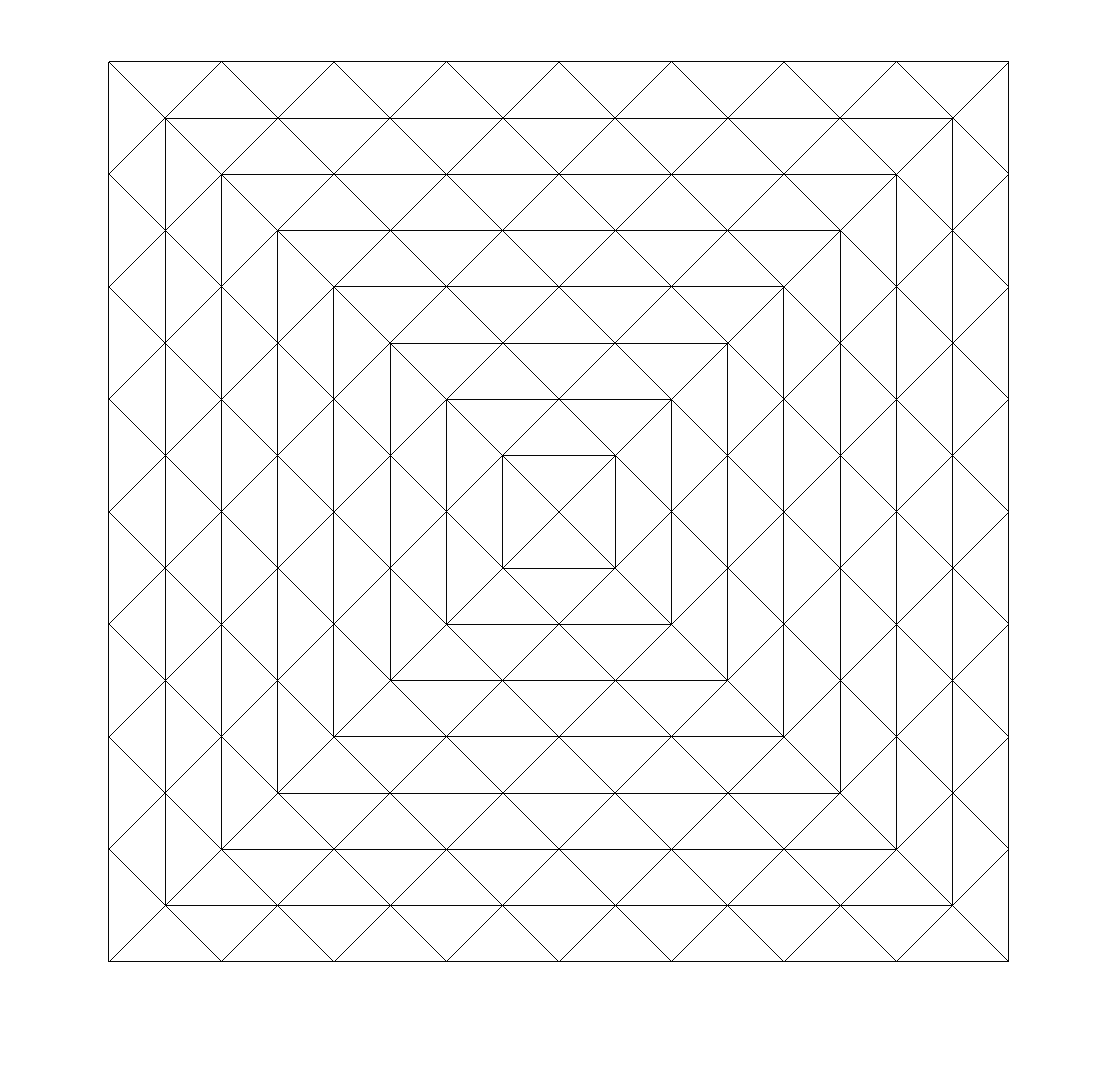}
	\end{minipage}
	\begin{minipage}[t]{0.3\textwidth}
		\centering
		\includegraphics[width=5cm]{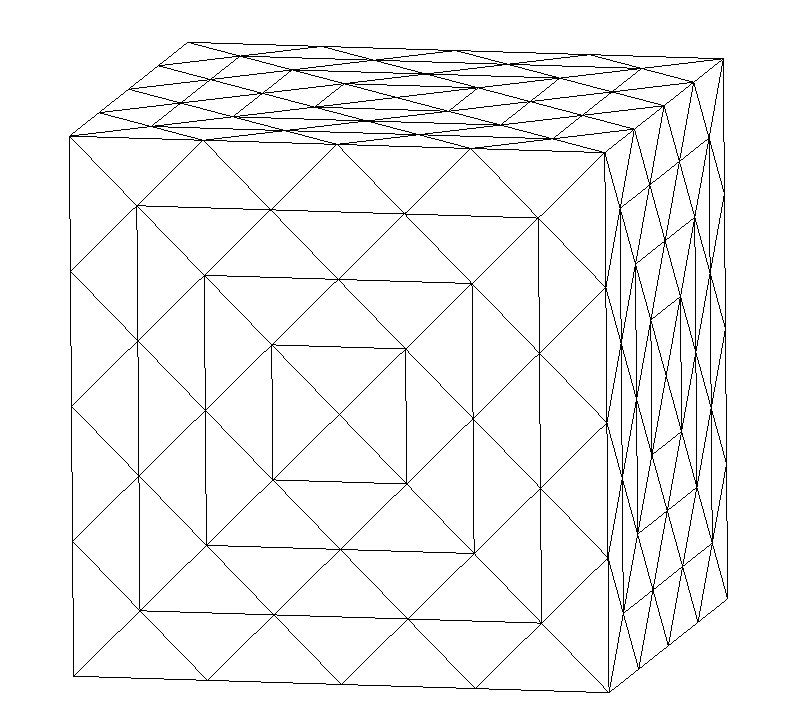}
	\end{minipage}
	\caption{an interval, a triangle and a tedrohedron partition}
	\label{typicalfemgrid123}
\end{figure}

A finite element space is defined in association with a simplicial
finite element grid ${\cal T}_h\subset \Omega$.  A simplicial finite
element grid ${\cal T}_h $ consists of a set of simplexes $\{\tau_k\}
$ and the corresponding set of nodal points is denoted by ${\cal
	N}_h$. For a given grid $\mathcal T_h$, the corresponding finite
element space is given by

\begin{equation}
	\label{femspace}
	V_h=\{v\in C(\Omega):   v \mbox{ is linear  on each element } \tau_k \in \mathcal T_h\}.
\end{equation}

Given $x_i\in \mathcal N_h$,  it is easy to see that there exists a unique
function $\phi_i\in V_h$, known as the nodal basis function, such that
\begin{equation}
	\label{phi-i}
	\phi_i(x_j)=\delta_{ij}, \quad x_j\in {\cal N}_h.   
\end{equation}
A typical profile of $\phi_i$ is shown in Fig.~\ref{fig:basis} for
$d=1$ and $d=2$.
\begin{figure}[htbp]
	\centering
	\begin{minipage}[t]{0.45\textwidth}
		\centering
		\includegraphics[width=6cm]{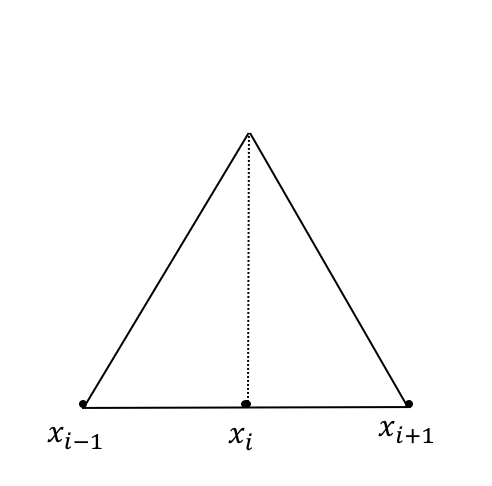}
	\end{minipage}
	\begin{minipage}[t]{0.45\textwidth}
		\centering
		\includegraphics[width=7cm]{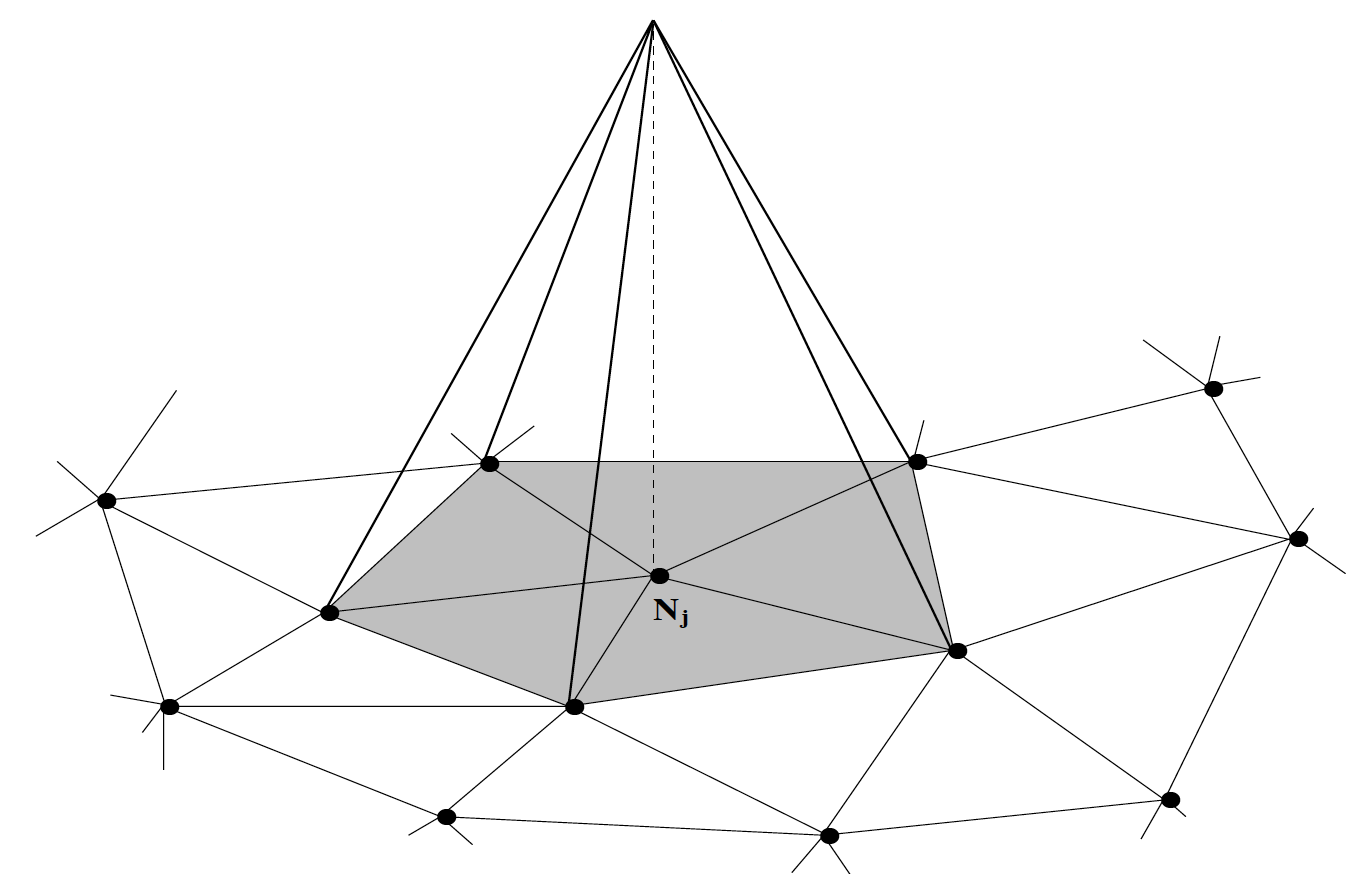}
	\end{minipage}
	\caption{The basis function in 1D and 2D}
	\label{fig:basis}
\end{figure}

Obviously any $v\in V_h$ can be uniquely represented in terms of these
nodal basis functions:
\begin{equation}
	\label{v-basis}
	v(x) = \sum_{i=1}^N \nu_i\phi_i(x),
\end{equation}
where $N$ is the degrees of freedom.

Given $x_i\in {\cal N}_h$, 
let $N(i)$
denote all the indices $j$ such that $\tau_j$ contains the
nodal point $x_i$, namely 
$$
N(i)= \{j: x_i\in \tau_j\}, 
$$ 
and $k_h$ denote the maximum number of
neighboring elements in the grid
\begin{equation}
	\label{kh}
	k_h\equiv d(\mathcal T_h)=\max_{x_i\in {\cal N}_h}| N(i)|.   
\end{equation}
Let $G(i)$ denote the support of the nodal basis $\phi_i$:
$$
G(i)=\bigcup_{k\in N(i)} \tau_k.
$$
We say that the grid $\mathcal T_h$ is locally convex if $G(i)$ is convex  for
each $i$.  

We proceed next to demonstrate how a finite element function can
be represented by a ReLU DNN. Our derivation
and analysis are based on the representation of the finite element
function as a linear combination of basis functions as follows.

\subsection{DNN representation of finite element functions}
As an illustration, we will now demonstrate how a linear finite
element function associated with a locally convex grid $\mathcal T_h$
can be represented by a ReLU DNN.  For more general grids, we refer to
Remark~\ref{rk:convexity} and \S\ref{sec:cpwl}.

Thanks to \eqref{v-basis}, it suffices to show that each basis
function $\phi_i$ can be represented by a ReLU DNN. 
We first note that the case where $d=1$ is trivial as the basis
function $\phi_i$ with support in $[x_{i-1},
x_{i+1} ]$ can be easily written as
\begin{equation}
	\label{1d-basis}
	\phi_i(x) = \frac{1}{h_{i-1}}\relu(x-x_{i-1}) -(\frac{1}{h_{i-1}}+\frac{1}{h_i})\relu(x-x_i) +\frac{1}{h_i}\relu(x-x_{i+1}),
\end{equation}
where $h_i = x_{i+1} - x_i$.

In order to consider the cases where $d>1$, we first prove the following lemma. 
\begin{lemma} \label{basismaxmin} Given $x_i\in {\cal N}_h$, if $G(i)$ is
	convex, then the corresponding basis function can be written as
	\begin{equation}
		\label{i-maxmin}
		\phi_i(x)=\max\bigg\{0,\min\limits_{k\in N(i)}g_k(x)\bigg\},
	\end{equation}
	where, for each $k\in N(i)$, $g_k$ is the global linear function such
	that $g_k=\phi_i$ on $\tau_k$.
\end{lemma}
\begin{proof} To show \eqref{i-maxmin} holds for all $x\in
	\mathbb{R}^d$, we first consider the case $x\in G(i)$, namely $x\in
	\tau_{k_0}$ for some $k_0\in N(i)$.  Thus
	\begin{equation}
		\label{phi-g}
		\phi_i(x) = g_{k_0}(x)\ge0.    
	\end{equation}
	Let $P_k$ be the hyperplane that passes through the $d-1$ subsimplex
	(of $\tau_k$) that does not contain $x_i$ (see the left figure in Figure \ref{fig:prooflemmamaxmin}).  Since 
	$G(i)$ is convex by assumption, all points in $\tau_{k_0}$ should be on the same
	side of the hyperplane $P_k$. As a result, for all $k\in N(i)$, 
	$$
	g_k(y)\ge 0 \equiv g_{k_0}(y), \quad y\in P_{k_0}\cap\tau_{k_0}. 
	$$
	By combining the above inequality with the following obvious
	inequality that 
	$$
	g_{k}(x_i) =1 \ge 1 = g_{k_0}(x_i), \quad k\in N(i),
	$$
	and the fact that all $g_k$ are linear, we conclude that 
	$$
	g_k(y)\ge g_{k_0}(y), \quad \forall y \in \tau_{k_0},k\in N(i).
	$$
	In particular
	$$
	g_k(x)\ge g_{k_0}(x), \quad k\in N(i).
	$$
	This, together with \eqref{phi-g}, proves that \eqref{i-maxmin} holds for all $x\in G(i)$. Thus
	\begin{equation}\label{eq:lem31:1}
		\max\bigg\{ 0,\min_{k\in N(i)}g_k(x) \bigg\}=g_{k_0}(x).
	\end{equation}

	\begin{figure}[htbp]
		\centering
		\begin{minipage}[t]{0.45\textwidth}
			\centering
			\includegraphics[width=6cm]{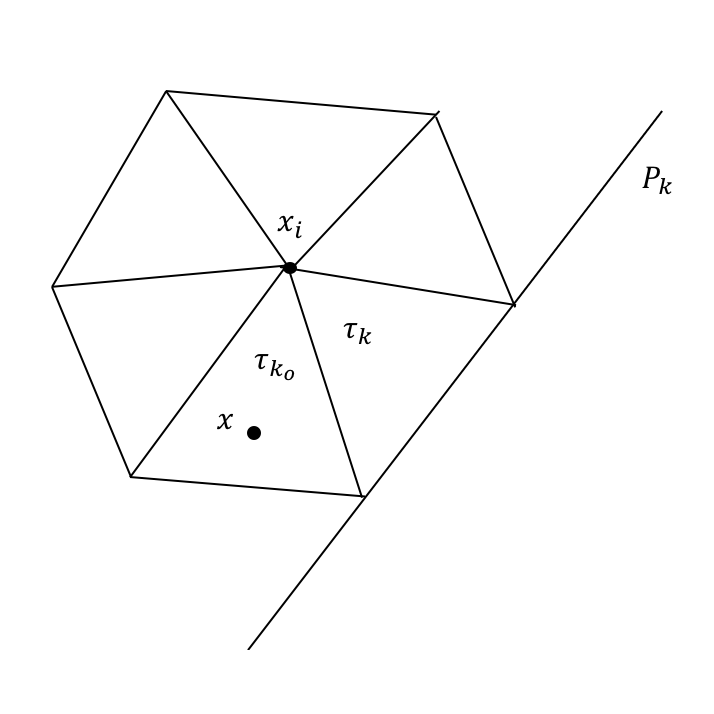}
		\end{minipage}
		\begin{minipage}[t]{0.45\textwidth}
			\centering
			\includegraphics[width=6cm]{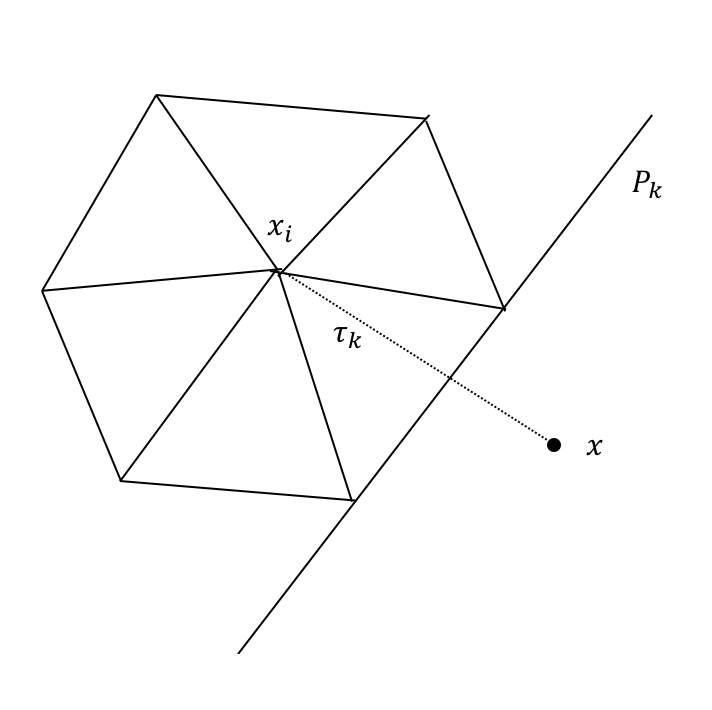}
		\end{minipage}
		\caption{Left: $x\in G(i)$, right: $x\notin G(i)$}
		\label{fig:prooflemmamaxmin}
	\end{figure}

	On the other hand, if $x\notin G(i)$, there exists a $\tau_k\subset
	G(i)$ such that $\tau_k$ contains a segment of the straight line that
	pass through $x$ and $x_i$(see the right figure in Figure \ref{fig:prooflemmamaxmin}).  Again let $P_k$ be the hyperplane
	associated with $\tau_k$ as defined above.  We note that $x$ and $x_i$ are
	on the different sides of $P_k$. Since 
	$$ g_k(x_i)\ge 0,\quad g_k(y) = 0,\quad y\in P_k,
	$$
	we then have
	$$
	\min_{k\in N(i)} g_k(x) \le g_k(x) \le 0,
	$$
	which implies
	$$
	\max\bigg\{ 0,\min_{k\in N(i)} g_k(x) \bigg\} = 0 = \phi_i(x),\qquad x\notin G(i).
	$$
	This finishes the proof of Lemma~\ref{basismaxmin}.
\end{proof}

\begin{remark}\label{rk:convexity} If $G(i)$ is not convex, we could also write the basis function as some max-min functions. But the form of max-min function is not as simple as the case where $G(i)$ is convex, and it depends on the shape of the support of the basis function. In some cases, we can write the basis function as the max-min-max form if $G(i)$ is a special non-convex set.   \end{remark}

We are now in a position to state and prove the main result in this section. 
\begin{theorem}\label{fem-dnn}
	Given a locally convex finite element grid ${\cal T}_h$, any linear
	finite element function with $N$ degrees of freedom, can be written
	as a ReLU-DNN with at most $k = \lceil\log_2 k_h\rceil+1$ hidden layers and at most $\mathcal{O}(k_hN)$ number of the neurons.
\end{theorem}
\begin{proof}
	We have the following identity,
	\begin{equation}\label{min2terms}
		\min\{a,b\} = \frac{a+b}{2} -  \frac{|a-b|}{2}= v\cdot \relu(W\cdot [a, b]^T),
	\end{equation}
	where 
	$$
	v = \frac{1}{2}[1,-1,-1,-1],\qquad W = \left[
	\begin{matrix}
	&1  & 1\\
	&-1 &-1\\
	&1  &-1\\
	&-1 & 1\\
	\end{matrix}
	\right].
	$$
	By Lemma~\ref{basismaxmin}, the basis function $\phi_i(x)$ can be written as:
	$$
	\phi_i(x) = \max\bigg\{ 0,\min_{k\in N(i)}g_k(x)\bigg\}.
	$$ 
	For convenience, we assume that
	$$
	N(i) = \{ r_1,r_2,...r_{|N(i)|} \}.
	$$
	Then we have
	\begin{equation*}
		\begin{aligned}
			\min\limits_{k\in N(i)}g_k(x)	&= \min\big\{g_{r_1}(x),...,g_{r_{|N(i)|}}(x)\big\}\\
			&= \min\bigg\{ \min\{g_{r_1},...,g_{r_{\lceil|N(i)|/2 \rceil}}\},\min\{ g_{r_{\lceil|N(i)|/2 \rceil+1}},...,g_{r_{|N(i)|}} \} \bigg\}\\	
			&= 	v\cdot \relu(W\cdot
			\begin{bmatrix}
				\min\{g_{r_1},...,g_{r_{\lceil|N(i)|/2 \rceil}}\}\\
				\min\{ g_{r_{\lceil|N(i)|/2 \rceil+1}},...,g_{r_{|N(i)|}} \} 
			\end{bmatrix} ).
		\end{aligned}
	\end{equation*}
	According to this procedure, we get the minimum of $|N(i)|$ terms by splitting them in two, each taking the minimum over at  most $\lceil|N(i)|/2 \rceil$ terms. This contributes to one ReLU hidden layer. Then we can further split the terms 
	$$
	\min\{g_{r_1},...,g_{r_{\lceil|N(i)|/2 \rceil}}\},\quad
	\min\{ g_{r_{\lceil|N(i)|/2 \rceil+1}},...,g_{r_{|N(i)|}} \}
	$$
	until all the minimum functions contain only 1 or 2 terms.
	\begin{enumerate}
		\item If there is one term 
		$$
		\min \{a\} = a.
		$$
		\item If there are two terms
		$$
		\min\{a,b\} = v\cdot \relu(W\cdot [a,b]^T).
		$$
	\end{enumerate}
	which is also a ReLU DNN with $1$ hidden layer.  So we can write a basis function as a $1+\lceil\log_2 k_h\rceil$-hidden-layer DNN. Considering the binary-tree structure, a $k$-layer full binary-tree has $2^{k}-1$ nodes.  We can see the number of neurons is at most
	$$
	\mathcal{O}(2^k)  = \mathcal{O}(2^{1+\lceil\log_2 k_h\rceil}) = \mathcal{O}(k_h).
	$$
	By (\ref{v-basis}), the piecewise linear function can be represented
	as a DNN with $k = 1 + \lceil \log_2 k_h \rceil$ hidden layers. The
	number of neurons is at most $\mathcal{O}(k_hN)$. 
\end{proof}

We now consider a special class of the so-called shape regular finite
element grid $\mathcal T_h$ which satisfies
\begin{equation}
	\label{shape-regular}
	\kappa_1\le \frac{r_\tau}{R_\tau}\le \kappa_2,
	\quad\forall\tau\in\mathcal T_h,
\end{equation}
for some constants $\kappa_1$ and $\kappa_2$ independent of $h$ and
$d$, where $r_\tau$ ($R_\tau$)  is the radius of the largest (smallest) ball
contained in (containing) $\tau$. 
\begin{corollary}\label{fem-dnn-cont}
	Given a locally convex and shape regular finite element grid ${\cal T}_h$, any linear
	finite element function with $N$ degrees of freedom(DOFs), can be written
	as a ReLU-DNN with at most $\mathcal O(d)$ hidden layers. The
	number of neurons is at most $\mathcal{O}(\kappa^dN)$ for some
	constant $\kappa\ge 2$ depending on the shape-regularity of
	$\mathcal T_h$.  The number of non-zero parameters is at most $\mathcal{O} (d\kappa^dN)$.
\end{corollary}
We note that, using the approach described in this section, a finite
element function with $N$ DOFs can be represented by a
DNN with $\mathcal O(N)$ number of weights. This property is expected
to be useful when DNNs are used in adaptive mesh-less or vertex-less
numerical discretization methods for partial differential equation,
which is a subject of further study.

\subsection{Comparison of error estimates in adaptive finite element
	and DNN methods}
Error estimates for adaptive finite element methods are well
studied in the literature. For example, an appropriately adapted
linear finite element function with $\mathcal{O} (N)$ DOFs is proved
to admit the following error estimate:
\begin{equation}
	\label{femerror1}
	\inf_{v\in V_h}\|u- v\|_{0,2, \Omega} \le C N^{-\frac{2}{d}} |u|_{2,\frac{2d}{d+2},\Omega},
\end{equation}
if $u \in W^{2,\frac{2d}{d+2}}(\Omega)$ and $v$ is the
interpolation based on the adapted finite element grid.  More details
can be founded in \cite{nochetto2011primer, devore1998nonlinear}.

For a shallow network ${\rm DNN}_1$ with $\mathcal O(N)$ DOFs
(i.e. $\mathcal{O} (\frac{N}{d})$ neurons), we have the next error
estimate in \eqref{Barron0} as
\begin{equation}
	\min_{v \in {\rm DNN}_{1}^\frac{N}{d}(\Omega)} \|u - v\|_{0, 2, \Omega} \lesssim
	|\Omega|^{1/2} \left(\frac{N}{d}\right)^{-{1\over2}}
	\int_{\mathbb R^d}|\omega|_\Omega|\hat u(\omega)|d\omega. 
\end{equation}
In comparison, an adaptive finite element function with the same order of
$\mathcal O(N)$ DOFs can only have convergence rate of order $\mathcal
O(N^{-\frac{2}{d}})$.

As will be shown in \S \ref{sec:minilayers}, shallow neural
networks ${\rm DNN}_1$ (namely with only one hidden layer) cannot
recover a linear finite element function in general, but may potentially 
lead to better asymptotic accuracy as the dimension $d$ gets larger. 

One idea that may help us to understand is that the
shallow network is a kind of $N$-term or basis selection
(\cite{devore1998nonlinear} )approximation scheme with $\{\sigma(w_i
x+b_i)\}_{i=1}^N$ as the basis functions (as shown in
Theorem~\ref{linearindep}),  similar to using $\{sin(nx)\}_{n=1}^N$ as the
basis functions in Fourier approximation or some others in wavelets.

For deep ReLU neural networks, our connections of FEM and ReLU DNNs in
this section help us to construct a special ReLU DNN models with depth
$\mathcal{O}(d)$ and parameters $\mathcal O(dk_dN)$ for $\mathcal
O(N)$ DOFs. By using the approximation result for adaptive FEM, DNN
approximation $u_{DNN}$ for special structure with $\mathcal O(N)$ DOFs
can get
\begin{equation}\label{errordnnelem}
	\|u-u_{DNN}\|_{0,2, \Omega}\lesssim  \left(\frac{N}{dk_d}\right)^{-\frac{2}{d}}|u|_{1,\frac{2d}{2+d}, \Omega} \lesssim  (N)^{-\frac{2}{d}}|u|_{1,\frac{2d}{2+d}, \Omega},
\end{equation}
and $k_d = \mathcal{O}(\kappa^d)$.  This shows that there exists some special deep ReLU DNN structure which is at least as good as adaptive FEM.

\section{LFE can not be recovered by ${\rm DNN}_1$ for $d\ge 2$}\label{sec:FEM_DNN1}
\label{sec:minilayers}
In the previous section, we show that a finite element function can be
represented by a ReLU DNN with $\log_2 k_h+1$ hidden layers. 

In view of Lemma \ref{dnn-cpwl} and the fact that ${\rm{DNN}_J}
\subseteq {\rm{DNN}_{J+1}} $, it is natural to ask that how many
layers are needed at least to recover all linear finite element
functions in $\mathbb{R}^d$.  In this section, we will show that 
\begin{equation}\label{key}
	J_d \ge 2, \quad \text{if} \quad d\ge 2,
\end{equation}
where $J_d$ is the minimal $J$ such that all linear finite element
functions in $\mathbb R^d$ can be recovered by ${\rm DNN}_J$.

In particular, we will show the following theorem.
\begin{theorem}\label{lowerbound}
	If $\Omega\subset \mathbb R^d$ is either 
	a bounded domain or $\Omega=\mathbb{R}^d$,  
	${\rm DNN}_1$ can not be used to recover all linear finite element
	functions in $\Omega$. 
\end{theorem}
\begin{proof}
	We prove it by contradiction. Let us assume that for any continuous
	piecewise linear function $f: \Omega \to \mathbb{R} $, we can find
	finite $N \in \mathbb{N}$, $w_i \in \mathbb{R}^{1,d}$ as row vector
	and $\alpha_i, b_i, \beta \in \mathbb{R}$ such that
	$$
	f =  \sum_{i=1}^N \alpha_i \relu(w_i  x +b_i) + \beta,
	$$
	with $f_i = \alpha_i \relu(w_i x +b_i)$, $\alpha_i \neq 0$ and $w_i
	\neq 0$.  Consider the finite element functions, if this one hidden
	layer ReLU DNN can recover any basis function of FEM, then it can
	recover the finite element space.  Thus let us assume $f$ is a locally
	supported basis function for FEM, i.e. ${\rm{supp}}(f)$ is bounded,
	where
	$$
	{\rm{supp}(f)} := \overline{\{ x \in \Omega ~|~ f(x) \neq 0\} }.
	$$
	Furthermore, if $\Omega$ is a bounded domain, we assume that 
	\begin{equation}\label{distcondi}
		d({\rm supp}(f), \partial \Omega) > 0,
	\end{equation}with 
	$$
	d(A, B) = \inf_{x\in A, y\in B} \|x-y\|,
	$$ 
	as the distance of two closed sets. 
	
	A more important observation is that $\nabla f: \Omega \to
	\mathbb{R}^d$ is a piecewise constant vector function. The key
	point is to consider the discontinuous points for $g := \nabla
	f = \sum_{i=1}^N \nabla f_i$.
	
	For more general case, we can define the set of discontinuous points of a function by
	$$
	D_{g} := \{x \in \Omega~|~ x ~ \text{is a discontinuous point of} ~ g\}.
	$$
	Because of the property that 
	\begin{equation}\label{eq:disfun}
		D_{f+g} \supseteq D_{f} \cup D_{g} \backslash (D_{f} \cap D_{g}),
	\end{equation}
	we have
	\begin{equation}\label{eq:dis_fn}
		D_{\sum_{i=1}^N g_i} \supseteq \bigcup_{i=1}^N D_{g_i} \backslash \bigcup_{i\neq j}\left( D_{g_i}\cap D_{g_j} \right).
	\end{equation}
	Note that
	\begin{equation}\label{eq:def_gi}
		g_i = \nabla f_i(x) =  \nabla \left( \alpha_i \relu(w_i  x +b_i)  \right) =\left(\alpha_iH(w_i  x +b_i)\right)w_i \in \mathbb{R}^d,
	\end{equation}
	for $i=1:N$ with $H$ be the Heaviside function defined as: 
	$$
	H(x) = \begin{cases}
	0 &\text{if} ~ x \le 0, \\
	1 &\text{if} ~ x > 0.
	\end{cases}
	$$ 
	This means that 
	\begin{equation}\label{eq: D_gi}
		D_{g_i} = \{ x ~|~ w_i  x + b_i = 0\}
	\end{equation}
	is a $d-1$ dimensional affine space in $\mathbb{R}^d$.

	Without loss of generality, we can assume that 
	\begin{equation}\label{eq:assumD_gi}
		D_{g_i} \neq D_{g_j}.
	\end{equation}
	When the other case occurs, i.e. $D_{g_{\ell_1}} = D_{g_{\ell_2}} = \cdots= D_{g_{\ell_k}}$, by the definition of $g_i$ in \eqref{eq:def_gi} and $D_{g_i}$ in \eqref{eq: D_gi} , 
	this happens if and only if there is a row vector $(w, b)$ such that
	\begin{equation}\label{eq:Dfcondition}
		c_{\ell_i}\begin{pmatrix}
			w &
			b
		\end{pmatrix} =  
		\begin{pmatrix}
			w_{\ell_i} &
			b_{\ell_i}
		\end{pmatrix},
	\end{equation}
	with some $c_{\ell_i} \neq 0$ for $i = 1:k$.  We combine those $g_{\ell_i}$ as
	\begin{align}\label{mergeH}
		\tilde g_{\ell} &= \sum_{i=1}^k g_{\ell_i} = \sum_{i=1}^k \alpha_{\ell_i} H(w_{\ell_i}  x + b_{\ell_i}) w_{\ell_i}, \\
		&= \sum_{i=1}^k \left( c_{\ell_i}\alpha_{\ell_i} H\left(c_{\ell_i}(w  x + b)\right) \right) w, \\
		&=\begin{cases}
			\left(\sum_{i=1}^k  c_{\ell_i}\alpha_{\ell_i} H(c_{\ell_i}) \right) w  \quad &\text{if} \quad w x + b > 0,\\
			\left(\sum_{i=1}^k  c_{\ell_i}\alpha_{\ell_i} H(-c_{\ell_i}) \right) w  \quad &\text{if} \quad w x + b \le 0.\\
		\end{cases}
	\end{align}
	Thus, if 
	$$
	\left(\sum_{i=1}^k  c_{\ell_i}\alpha_{\ell_i} H(c_{\ell_i}) \right)  = \left(\sum_{i=1}^k  c_{\ell_i}\alpha_{\ell_i} H(-c_{\ell_i}) \right),
	$$
	$\tilde g_\ell$ is a constant vector function, that is to say $D_{\sum_{i=1}^k g_{\ell_i}} = D_{\tilde g_\ell} = \emptyset$. 
	Otherwise, $\tilde g_\ell$ is a piecewise constant vector function with the property that 
	$$
	D_{\sum_{i=1}^k g_{\ell_i}} = D_{\tilde g_\ell} = D_{g_{\ell_i}} = \{ x ~|~ w x + b = 0\}.
	$$
	This means that we can use condition \eqref{eq:Dfcondition} as an equivalence relation and split $\{g_i\}_{i=1}^N$ into some groups, and we can combine those $g_{\ell_i}$ in each group as what we do above. After that, we have
	$$
	\sum_{i=1}^N g_i = \sum_{\ell=1}^{\tilde N} \tilde g_{\ell},
	$$
	with $D_{\tilde g_s} \neq D_{\tilde g_t}$.
	Finally, we can have that $D_{\tilde g_s} \cap D_{\tilde g_t}$ is an empty set or a $d-2$ dimensional affine space in $\mathbb{R}^d$.
	Since
	$\tilde N \le N$ is a finite number, 
	$$
	D := \bigcup_{i=1}^N D_{\tilde g_\ell} \backslash \bigcup_{s\neq t}\left( D_{\tilde g_s}\cap D_{\tilde g_t} \right)
	$$
	is an unbounded set. 
	\begin{itemize}
		\item If $\Omega = \mathbb{R}^d$,
		$$
		{\rm supp(f)} \supseteq D_{g} = D_{\sum_{i=1}^N g_i} = D_{ \sum_{\ell=1}^{\tilde N} \tilde g_{\ell}} \supseteq D,
		$$ is contradictory to the assumption that $f$ is locally supported.
		\item If $\Omega$ is a bounded domain, 
		$$
		d(D, \partial \Omega) = 
		\begin{cases}
		s > 0 \quad &\text{if}\quad  D_{\tilde g_i} \cap \Omega = \emptyset, \forall i\\
		0 \quad &\text{otherwise}.
		\end{cases}
		$$
		Note again that all $D_{\tilde g_i}$'s are $d-1$ dimensional affine spaces, while $D_{\tilde g_i} \cap D_{\tilde g_j}$ is either an empty set or a d-2 dimensional affine space. 
		If $d(D, \partial \Omega) > 0$, this implies that $\nabla f$ is continuous in $\Omega$, which contradicts the  assumption that $f$ is a basis function in FEM.
		If $d(D, \partial \Omega) = 0$, this contradicts the previous assumption in \eqref{distcondi}.
	\end{itemize}
	Hence ${\rm DNN}_1$ cannot recover any piecewise linear function in $\Omega$ for $d \ge 2$.
\end{proof}

Following the proof above, we can prove Theorem~\ref{linearindep}.

\begin{proof}
	If $(w_i, b_i)$ and $(w_j, b_j)$ are linearly independent for any $i \neq j$, we know
	that the set of discontinuous points for any nontrivial combinations of $\nabla_x \relu(w_i  x + b_i)$ cannot be empty. 
	So, this is contradictory to
	$$
	f = \sum_{i=1}^m \alpha_i \relu(w_i x+b_i) = 0,
	$$ 
	since $D_{\nabla f} \neq \emptyset$ where $f(x)$ is a combination of $ \{\relu(w_i  x + b_i)\}$. 
\end{proof}

This shows that despite it has the so-called universal approximation
properties \cite{hornik1989multilayer, cybenko1989approximation},
shallow network is not enough in the case of recovering all CPWL
functions.  More precisely, although the shallow ReLU DNNs are CPWL
functions themselves and can approximate any CPWL functions with any
accuracy, there are some CPWL functions they cannot represent
exactly. As an example, a local basis function in FEM with compact
support and some other simple conditions cannot be represented by ReLU
DNNs with one hidden layer for dimensions greater than $2$.

As for the upper bound, Theorem \ref{main0} in \cite{arora2016understanding} provides us with one answer.

\begin{corollary}
	\begin{equation}
		2 \le J_d \le \lceil\log_2(d+1)\rceil.
	\end{equation}
	This also indicates that $\lceil\log_2(d+1)\rceil$ is ``optimal" for $d=2,3$.
\end{corollary}

\section{General CPWL as a ReLU DNN}\label{sec:cpwl}
In the previous sections, we present a special approach to represent a
linear simplicial finite element function by a ReLU DNN.  
In this section, we discuss a general approach to represent a general CPWL by
a ReLU DNN, which is introduced in \cite{arora2016understanding}.  In comparison with the special approach in \S\ref{sec:FEM_DNN}, this general approach gives a ReLU DNN with relatively fewer
layers but significantly more number of neurons.

\subsection{The main result}
Assume that $f:\mathbb{R}^d\to\mathbb{R}$ is a continuous function
that are piecewise linear on $m$ subdomains
$$
\Omega_i, \quad i=1:m.
$$
Namely, on each $\Omega_i$, $f$ is a linear function:
$$
f(x)=f_i(x)=a_i\cdot x+b_i, \quad x\in \Omega_i,
$$
with some $a_i \in \mathbb{R}^d$ and $b_i \in \mathbb{R}$.
\begin{lemma}
	\label{unique}
	There are $M$ number of subdomains, 
	$$
	\tilde \Omega_k, \quad k=1:M,
	$$
	such that
	$$
	f_i-f_j \mbox{ does not change sign on each } \tilde\Omega_k, k=1:M. 
	$$
	Furthermore
	$$
	m\le M\le m!
	$$   
\end{lemma}

\begin{proof}
	Because $m$ is finite, so there must exist $M$ number of subdomains,
	$$
	\tilde \Omega_{k}, \quad k = 1:M
	$$
	such that 
	$$
	f_i-f_j \mbox{ does not change sign on each } \tilde\Omega_k, k=1:M. 
	$$
	Then we proceed to estimate $M$. On the one hand, we have $m$
	pieces linear functions, so
	$$
	M\ge m.
	$$
	On the other hand, on each $\tilde \Omega$, we have the same
	rearrangement in ascending order of the values of the $m$ linear
	functions. There are at most $m!$ possible rearrangements. Then we show that for any $\tilde \Omega_i$ and $\tilde \Omega_j$,
	they must be the same subdomain if they have the same rearrangement in
	ascending order. If not, there must exist a boundary formed
	by two linear functions $l_p$ and $l_q$, and $\tilde \Omega_i$ and
	$\tilde \Omega_j$ must be on the different sides of the boundary. The
	order of the $l_p$ and $l_q$ must be opppsite on $\tilde \Omega_i$ and
	$\tilde \Omega_j$, which leads to a
	contradiction. So
	$$
	M\le m!.
	$$
\end{proof}
There is an important theorem named the lattice representation theorem
for CPWL functions in $\mathbb{R}^d$, more details can be founded in
\cite{tarela1999region}.

\begin{theorem}
	\label{lattice}
	For every continuous piecewise linear function
	$f:\mathbb{R}^n\to\mathbb{R}$ with finite pieces defined by the
	distinct local linear functions $l_i$, $1\le i\le m$ and
	$\{\tilde{\Omega}_k\}_{k=1}^M$ be the unique-order subdomains.Then
	there exist finite non-empty subsets of $\{1,2,\dots,m\}$, say $s_k$,
	$1\le k\le M$, such that
	\begin{equation}
		f(x)=\max_{1\le k\le M}\{\min_{i\in s_k} l_i\}.
	\end{equation}
\end{theorem}

For the relationship between ReLU DNNs and CPWL functions, we have the next theorem with some estimation.
\begin{theorem}\label{main0}
	A continuous function $f:\mathbb{R}^d\to\mathbb{R}$ that are piecewise linear on $m$ subdomains can be represented by a ReLU DNN.
	Furthermore, 
	\begin{enumerate}
		\item the number of  hidden layers  is bounded by 
		\begin{equation}
			\label{layer}
			N_{\rm layer}\le \lceil \log_2(d+1)\rceil.      
		\end{equation}
		\item the number of neurons 
		\begin{equation}
			\label{neurons}
			N_{\rm neuron}=
			\left\{
			\begin{array}{ll}
				\mathcal  O\left(d2^{mM+(d+1)(m-d-1)}\right) & \mbox{ if } m\ge d+1,\\     
				\mathcal O\left(d2^{mM}\right) & \mbox{ if } m< d+1.
			\end{array}
			\right.
		\end{equation}
		
		here $M$, satisfying $m\le M\le m!$, is the number of subdomains
		as defined in Lemma~\ref{unique}. 
	\end{enumerate}
\end{theorem}
In Theorem~\ref{main0}, a relatively shallow $\relu$ DNN is
introduced. For any CPWL function in $\mathbb{R}^d$, we can use a
$\relu$ DNN with at most $\log_2\lceil(d+1)\rceil$ hidden layers to
represent it. Moreover, the second part of Theorem~\ref{main0} gives
us size estimates on such neural networks.

The main result in Theorem~\ref{main0} is not new, which can be found in
\cite{arora2016understanding}.  In the next subsection, we will give
an outline of the proof of Theorem~\ref{main0} and in particular to
derive the new estimate \eqref{neurons} on the number of neurons
needed in the DNN representation.  We will also discuss the
application of this theorem to the simplicial finite element space
in \S~\ref{sec:FEM}.

\subsection{On the proof of Theorem~\ref{main0}}\label{logn1} 
In this section, we give an outline of the proof of Theorem
\ref{main0}.  We will mainly follow the proof in
\cite{arora2016understanding} which is based on many relevant results
in existing literature such as
\cite{arora2016understanding,wang2005generalization}, but add some
detailed estimate of the number of neurons.

\begin{lemma}\label{bar-alpha}
	Let $f(\cdot)$, $g(\cdot)$ and $h(\cdot)$ be arbitrary functions
	from $\mathbb{R}^d\to\mathbb{R}$. If $\alpha \neq 1 $, let
	$\bar{\alpha}=\frac{1}{1-\alpha}$. Then the following identity is
	valid for all $x\in\mathbb{R}^d$,
	\begin{equation}
		\begin{aligned}
			\label{fir}
			\max\{f,g,\alpha g+h\}&=\max\{f,\max\{1,\alpha\}\max\{g-\bar{\alpha}h,0\}+\bar{\alpha}h\}\\
			&+\max\{f,\min\{1,\alpha\}\min\{g-\bar{\alpha}h,0\}+\bar{\alpha}h\}-\max\{f,\bar{\alpha} h\}.
		\end{aligned}
	\end{equation}
\end{lemma}

Moreover, we have:
\begin{equation}\label{key-reduce}
	\max\{f,g,\alpha g+h\}=	\sigma_1\max\{f,g,\bar{\alpha}h\}+\sigma_2\max\{f,\alpha g+h,\bar{\alpha}h\}+\sigma_3\max\{f,\bar{g}\}.
\end{equation}
here $\sigma_k\in\{1,-1\}$ and $\bar{g}$ is one of $g$, $\alpha g+h$ or $\bar{\alpha}h$.

\begin{lemma}\label{reduce}
	For any interger $L$ with $1\le n< L$, $c_0\in\mathbb{R}$ and arbitrary linear function $l_1(x),\dots,l_L(x)$ of $x\in\mathbb{R}^d$, there exist finite groups of $L-1$ linear functions, say $l(x,b_1(k))$, $\dots$,
	$l(x,b_{L-1}(k))$, $1\le k\le K$, and corresponding $c_k\in\mathbb{R}$, $\sigma_k\in\{1,-1\}$ such that
	
	\begin{equation}\label{goal}
		\max\{c_0,l_1,\dots,l_L\} = \sum_{k=1}^{K}\sigma_k\max\{c_k,l(x,b_1(k)),\dots,l(x,b_{L-1}(k))\}.
	\end{equation}
\end{lemma}

\begin{lemma}\label{size}
	Let $f_1,\dots,f_m:\mathbb{R}^d\to\mathbb{R}$ be functions that can each be represented by $\mathbb{R}^d\to\mathbb{R}$ ReLU DNNs with depth $k_i+1$ and size $s_i$. Then the function $f:=\max\{f_1,\dots,f_m\}$ can be represented by a ReLU DNN of depth at most $\max\{k_1,\dots.k_m\}+\lceil\log m\rceil+1$ and size at most $s_1+\dots+s_m+4(2m-1)$.
\end{lemma}

With all the lemmas above, now we can start to prove Theorem \ref{main0}.

\begin{proof}
	For any continuous piecewise linear function $f$ from $\mathbb{R}^d$ to $\mathbb{R}$, we have the following lattice representation by Theorem \ref{lattice}:
	$$
	f=\max_{1\le j\le M}\min_{i\in s_j}l_i.
	$$
	where $M$ is the number of subdomains as defined in Lemma~\ref{unique}, and $s_j\subset\{1,\dots,m\}$, $m$ is the number of distinct pieces of $f$. 
	
	Let $\Phi_j=\max\limits_{1\le t\le j}\{\min_{i\in s_t} l_i\},\ 1\le j\le M$, then 
	\begin{equation}\label{fenjie}
		f=\Phi_M=\max\{\Phi_{M-1},  \min_{i\in s_M} l_i\}.
	\end{equation}
	
	Since 
	\begin{equation*}
		\begin{aligned}
			\max\{f,\min_{1\le i\le m}l_i\}=&\sum_{i}\max\{f,l_i\}\\
			&+\dots+(-1)^{m+1}\max\{f,l_1,\dots,l_m\},
		\end{aligned}
	\end{equation*}
	
	we can write Equation~(\ref{fenjie}) as:
	\begin{equation}\label{step1}
		f=\sum_{n=1}^{K_1}\sigma_n^{(1)}\max\{\Phi_{M-1},\max_{i\in s^{(1)}_n}l_i\},
	\end{equation} 
	with $K_1\le 2^m-1$, $s^{(1)}_n\subset\{1,2,\dots,m\}$ and  $\sigma^{(1)}_n\in\{1,-1\}$.
	
	For each $\max\{\Phi_{M-1},\max_{i\in s^{(1)}_n}l_i\}$ in (\ref{step1}), we can continue this procedure for $M-1$ times. Then we have:
	\begin{equation}\label{final}
		f=\sum_{n=1}^{M'}\sigma_n\max_{i\in s'_n} l_i,
	\end{equation}
	here $s'_n\subset\{1,2,\dots,m\}$ and $M'\le(2^m-1)^{M}$. 
	
	Now that we write a piecewise linear function in form of
	(\ref{final}), in order to get the $\lceil
	\log_2(d+1)\rceil$-hidden-layer $\relu$ DNN, we need to do linear
	transformations to reduce the cardinality of $s'_n$ from $m$ to
	$n+1$. This can be done by Lemma \ref{reduce}.
	
	Following the procedures in Lemma~\ref{reduce} by
	\cite{wang2005generalization}, when reducing one cardinality
	of $s'_n$, one $\max_{i\in s'_n} l_i$ will become at most $2^{d+1}-1$ terms. If the cardinality is reduced from $m$ to $d+1$, then we need to repeat the whole
	procedure $m-d-1$ times. Hence for each $\max_{i\in
		s'_j}\{l_i\}$, in total we have at most $(2^{d+1}-1)^{m-d-1}$
	terms. Thus for
	$$
	f=\sum_{j=1}^{p}s_j\max_{i\in S_j} l_i,
	$$
	with $s_j\in\{-1,+1\}$ and $|S_j|\le d+1$, we have $p\le(2^m-1)^{M}(2^{d+1}-1)^{m-d-1}$.
	
	For each $\max_{i\in S_j} l_i$ with $|S_j|\le d+1$, Lemma \ref{size} can be used to get a ReLU DNN with $\lceil\log_2(d+1)\rceil$ hidden layers. Again by Lemma \ref{size}, the size is at most $2(d+1)+4(2(d+1)-1)=10d+6$. Adding these $\max_{i\in S_j} l_i$ together, we have the total size is at most $(10d+6)(2^m-1)^{M}(2^{d+1}-1)^{m-d-1}$, which is $\mathcal{O}(d2^{mM+(d+1)(m-d-1)})$.
	
	Note that if $m\le d+1$, we do not need to use Lemma ~\ref{reduce}, the size will be at most $\mathcal{O}(d2^{mM})$.
\end{proof}

The estimation in Theorem \ref{main0}  is a rough one, but still can provide some insights of this DNN representation. It can be seen that although the depth of this DNN is relatively shallow, the size of it might be extremely large, depending on the numbers of subdomains and distinct pieces.


\subsection{Linear finite element functions as DNN with $\lceil\log_2
	(d+1)\rceil$ hidden layers}
Given a locally convex finite element grid, now we have two different
ways to represent a linear finite element function. In this part, we
estimate the number of neurons if we write the function as a ReLU DNN
with at most $\lceil\log_2(d+1)\rceil$ hidden layers. Then we can
compare the sizes of two different approaches. Again we start with the
basis functions.
\begin{theorem}\label{basisfunctionsize}
	Given $x_i$, denote the corresponding basis function as $\phi_i$. If
	$G(i)$ is convex ,then the ReLU DNN with at most $\lceil
	\log_2(d+1)\rceil$ hidden layers has size at most
	$\mathcal{O}(d2^{(d+1)k_h})$.
\end{theorem}
\begin{proof}
	From Lemma~\ref{basismaxmin}, we know that
	$$
	\phi_i(x)=\max\bigg\{ 0,\min_{k\in N(i)} g_k(x) \bigg\}.
	$$
	For simplicity, let us further assume that 
	$$
	\phi_i=\max\bigg\{0,\min\{g_{r_1},\dots,g_{r_{|N(i)|}}\}\bigg\}.
	$$
	The first step is to write it as the linear combination fo max functions
	$$
	\begin{aligned}
	\phi_i=&\max\bigg\{0,\min\{g_{r_1},\dots,g_{r_{|N(i)|}}\}\bigg\},\\
	=&\sum_{k\in N(i)} \max\{0,g_k\}+\dots+(-1)^{|N(i)|+1}\max\{0,g_{r_1},\dots,g_{r_{|N(i)|}}\}.
	\end{aligned}
	$$
	Our goal is to make every term on the right hand side only take maximum over at most $d+1$ linear functions, here $d$ is the dimension.
	
	For any term with linear functions more than $d+1$, we need to use linear transformation to reduce this number. When reducing by one, one term will become at most $2^{d+1}-1$ terms. Thus $\max\{0,g_1,\dots,g_l\}$ will become at most $(2^{d+1}-1)^{l-d}$ when $l\ge d$.
	
	For any term with number of linear functions less or equal than $d+1$, it remains unchanged. The number of this kind of terms is
	$$
	C_{|N(i)|}^{1}+C_{|N(i)|}^{2}+\dots+C_{|N(i)|}^d.
	$$
	
	Then in total the number of terms should be
	$$
	N=\sum_{j=1}^{d}C_{|N(i)|}^j+\sum_{j=d+1}^{|N(i)|}C_{|N(i)|}^j(2^{d+1}-1)^{j-d}.
	$$
	
	Since any linear transformation $T:\mathbb{R}^d\to\mathbb{R}$ can be represented by a 2-layer ReLU DNN of size 2, for $\max\limits_{i\in S}\{l_i\}$ with the number of $|S|\le s$, it can be represented by a ReLU DNN of size at most $2s+4(2s-1)=10s-4$. So the total size is
	$$
	\begin{aligned}
	S&=\sum_{j=1}^{d}C_{|N(i)|}^j(10j+6)+(10d+6)\sum_{j=d+1}^{|N(i)|}C_{|N(i)|}^j(2^{d+1}-1)^{j-d},\\
	&= \mathcal{O}(d2^{(d+1)|N(i)|}),\\
	&=  \mathcal{O}(d2^{(d+1)k_h}).
	\end{aligned}
	$$
\end{proof}

\begin{corollary}
	Given a locally convex finite element grid $\mathcal{T}_h$, any linear finite element function in $\mathbb{R}^d$ with $N$ degrees of freedom, can be written as a DNN function with at most $\lceil \log_2(d+1)\rceil$ hidden layers and with size at most $\mathcal{O}(d2^{(d+1)k_h}N)$.
\end{corollary}

\begin{proof}
	According to Corollary~\ref{basisfunctionsize}, every basis function has a size independent of $N$, so the size of the DNN function with at most $(\lceil \log_2(d+1)\rceil)$ hidden layers is at most $\mathcal{O}(d2^{(d+1)k_h}N)$.
\end{proof}

By comparing the above results with Theorem~\ref{fem-dnn}, we can see
that although the DNN with $\lceil \log_2(d+1)\rceil$ hidden layers
has shallower depth, the number of neurons is much larger than the one
with $\lceil \log_2 k_h\rceil+1$ hidden layers.

\section{Low bit-width DNN models}\label{sec:specialstruc}
In this section, we will show the rationality of low bit-width models with respect to 
approximation properties in some sense by investigating that a special type of ReLU DNN
model can also recover all CPWL
functions. In~\cite{yin2018binaryrelax}, an incremental network
quantization strategy is proposed for transforming a general trained
CNN into some low bit-width version in which there parameters are all
zeros or powers of two. Mathematically speaking, low bit-width DNN model
is defined as:
\begin{equation}\label{set:lbd-dnn}
	{\rm DNN}_{k,l}^J :=  \{ f = \Theta^J \circ \sigma \circ \Theta^{J-1} \cdots \sigma\circ \Theta^0(x), \Theta^\ell_{i,j} \in \mathcal{Q}_{k,l}\},
\end{equation}
where $\sigma$ is the activation function and 
\begin{equation}\label{eq:Qkl}
	\mathcal{Q}_{k,l} := 2^k \times \{0, \pm 2^{1- 2^{l-2}}, \pm 2^{2 -2^{l-2}}, \ldots, \pm1\}.
\end{equation}
In~\cite{yin2018binaryrelax}, they introduce a closed projected formula for finding the optimal approximation of vector  $W^* \in \mathbb{R}^{m}$ in $\mathcal{Q}_{k,l}$. 
\begin{equation}\label{key}
	\min_{s, W} \|W - W^*\|^2_{F}, \quad \text{subject to} \quad W_i \in \mathcal{Q}_{k,l}.
\end{equation}
Under this closed form, they propose a projected gradient descent methods with respect to SGD to train a general R-FCN \cite{dai2016r} model for object detection.
They also find that 6-bit (i.e $b = 6$) model works almost the same with classical model in the object detection tasks. 

Then it comes the question: why can those kinds of models work?
More precisely, for classification or detection problems, can this model separate those data exactly? 
By our results in previous sections, we find a special family of ReLU DNN which has at most one general layers and all other layers with low bit-width parameters.
The results offer modification and theoretical explanation of the
existing low bit-width DNNs proposed in the literature.  

Here we try to explain why those low bit-width DNN model also work for classification problems to some extent. We have the following result:
\begin{theorem}
	Any continuous piecewise function can be represented by the next model:
	\begin{equation}\label{set:lbd-dnn-a}
		\widetilde{{{\rm DNN}}}_{0,3}^J := \{ f = \Theta^J \circ \relu \circ \Theta^{J-1} \cdots \relu\circ \Theta^0(x), \Theta^\ell_{i,j} \in \mathcal{Q}_{0,3}, \forall \ell \ge 1, \Theta^0_{i,j} \in \mathbb{R}\},
	\end{equation}
	with $ \mathcal{Q}_{0,3}$ defined in \eqref{eq:Qkl} and $J \ge \lceil\log_2(d+1)\rceil$. 
\end{theorem}

\begin{proof}
	Because of Theorem~\ref{main0}, we can rewrite any piecewise linear function as a ReLU DNN
	$$
	f(x) =  \Theta^{J_0} \circ \relu\circ \Theta^{J_0-1} \circ ... \circ \relu\circ \Theta^0 (x),
	$$
	with $J_0 \le \lceil\log_2(d+1)\rceil $. By (\ref{min2terms}), we know that
	\begin{enumerate}
		\item For each $\ell \ge 1$, we have
		$$
		b^\ell =0 ,
		$$
		and $W^\ell=\{w^\ell_{ij}\}$ with $w^\ell_{ij}\in\{0,\pm1,\pm1/2\}$.
		\item For $\ell=0$, where we have a ``fully connected layer'', the $W^0,b^0$ are determined by those linear functions.
	\end{enumerate} 
	Also note that $\widetilde{{{\rm DNN}}}_{0,3}^{J_0}\subseteq\widetilde{{{\rm DNN}}}_{0,3}^J$ if $J_0\le J$.
	This completes the proof.
	
\end{proof}

Although we have the universal approximation property for DNN with a single hidden layer, in which model the last layer $\theta^1 \in \mathbb{R}^{n^1 \times1}$ is still fully connected, this is a little bit different from the $\widetilde{{{\rm DNN}}}_{0,3}^J $ models defined above.



\section{Application to Numerical PDEs}\label{sec:appliPDE}
In this section, we discuss the application of DNNs to the numerical
solution of partial differential equations (PDEs). In most of our
discussion, we consider the following model problem:
\begin{equation}\label{poisson}
	\begin{split}
		-\Delta u = f,&\qquad x\in \Omega,\\
		\frac{\partial u}{\partial\nu} = 0, &\qquad x\in \partial\Omega,
	\end{split}
\end{equation}
here $\Omega\subset\mathbb R^d$ is a bounded domain. For simplicity of
exposition, we only consider Neuman boundary condition here.  As it is
done in the literature, special cares need to be taken for Dirichlet
boundary value problems, but we will not get into those (standard)
details.

The idea of using DNN for numerical PDEs can be traced back to 
\cite{lagaris1998artificial} where a collocation method is
used. 
Similar ideas have been explored by many different authors for
different types of PDEs.

For the model problem \eqref{poisson}, roughly speaking, the
collocation method amounts to the following least square problem:
\begin{equation}\label{collocation}
	\min_{\Theta} 
	\sum_{x_i\in\Omega}(-\Delta u_N(x_i,\Theta) - f(x_i)) ^2,
\end{equation} 
here $u_N(x,\Theta)$ is taken among the DNN function class in the
form of \eqref{compress-dnn} with a smooth activation function such as
sigmoidal function and $x_i$ are some collocation points. 

Recently, \cite{e2018deep} applied DNN for numerical PDE in the
Galerkin setting which amounts to the solution of the following energy
minimization problem:
\begin{equation}\label{galerkin}
	\min_{\Theta} \int_\Omega (\frac{1}{2}|\nabla u_N(x, \Theta)|^2 - fu_N(x,\Theta))dx
\end{equation}
Numerical experiments have demonstrated the potential of this
approach.  In the rest of this section, we will discuss a number of
aspects of this approach from both theoretical and practical
viewpoints.  In particular, we will discuss its relationship with two
popular finite element methods: adaptive finite element method and
moving grid method.

\subsection{The finite element method}
The finite element approximation to \eqref{poisson} can be written as 
\begin{equation}\label{fem}
	\min_{v\in V_h} \int_\Omega (\frac{1}{2}|\nabla v(x)|^2 - fv(x))dx,
\end{equation}
where $V_h$ is the finite element space as described in
\S\ref{sec:FEM}.

In the finite element setting, the optimization problem \eqref{fem} is
to find the coefficient $(\nu_i)$ as in \eqref{v-basis} for a given
finite element mesh $\mathcal T_h$.  Some more sophisticated versions
of the finite element method can be obtained by varying or optimizing
$\mathcal T_h$ so that more accurate finite element approximation can
be obtained.  Roughly speaking, there are two main approaches for
optimizing $\mathcal T_h$: one is the adaptive finite element method
and the other is the moving grid finite element method.

The adaptive finite element method is, roughly speaking, to vary
$\mathcal T_h$ by either coarsening or refining the grid.  One main
theoretical result is that a family of adapted grids $\mathcal T_h$
with $\mathcal O(N)$ degrees of freedom can be obtained so that the
corresponding adaptive finite element approximation $u_N$ satisfies
the following error estimate
\begin{equation}
	\label{femerror}
	|u- u_{h}|_{1,2, \Omega} \le C N^{-\frac{1}{d}} |u|_{2,\frac{2d}{d+2},\Omega}.
\end{equation}
We refer to \cite{nochetto2011primer, devore1998nonlinear} for
relevant details and its generalizations. 

One interesting observation is that the convergence rate $\mathcal
O(N^{-\frac{1}{d}} )$ in \eqref{femerror} deteriorate badly as $d$
increases.  Of course, error estimate in the form \eqref{femerror}
depends on which  Sobolev or Besov function classes that the solution
$u$ belongs to, namely what norms are used in the right hand side of
\eqref{femerror}.   But regardless what function classes for the
solution $u$,  no asymptotic error estimate seems to be known in the
literature that is better than $\mathcal O(N^{-\frac{1}{d}} )$.

The moving grid method is, on the other hand, to optimize $\mathcal
T_h$ by varying the location of grid points while preserving
topological structure of the grids (in particular the number of grid
ponts remain unchanged).  This approach proves to be effective in many
applications, see \cite{li2001moving,li2002moving}.  But there are
very few theories on the error estimate like \eqref{femerror} in the
moving grid method.

However, the $H^1$ approximation properties are still unclear even for
${\rm DNN}_1$.  \cite{xu2017deep} proves a similar result for $H^1$
error estimate for ${\rm DNN}_1^m(\Omega)$ with activation function
$\sigma(x) = \cos(x)$. For general activation functions, or just for
$\relu$, it is an open problem.

\subsection{DNN-Galerkin method}
The finite element methods discussed above, including adaptive method
and moving grid method, depend crucially on the underlying finite
element grids.  Numerical methods based on DNN, as we shall describe
now, are a family of numerical methods that require no grids at all.
This is reminiscent of the ``mesh-less method'' that have been much
studied in recent years
\cite{liu2002mesh,yagawa1996free,idelsohn2003meshless}.  But the
mesh-less method still requires the use of discretization points.  The
{\it DNN-Galerkin} method (as we shall call), namely the Galerkin
version of the DNN-element method such as \eqref{galerkin}, goes one
step further: it does not even need any discretization points!  It is
a totally point-free method!

Let us now give a brief discussion on the error estimate for the
DNN-Galerkin method.  We first recall a classic result by
\cite{barron1993universal} for a DNN with one hidden layer of
$\mathcal O(N)$ DOFs ( i.e. $\mathcal O(\frac{N}{d})$ neurons),
\begin{equation}
	\label{errorshallow0}
	\inf_{v\in {\rm DNN}_{1}^{\frac{N}{d}}(\Omega)}\|u- v\|_{0,2, \Omega} \lesssim
	\left(\frac{N}{d}\right)^{-\frac{1}{2}}
	C_u,
\end{equation}
here we have
\begin{equation}\label{C_uBarron}
	C_u := \int_{\mathbb R^d}|\omega|_{\Omega}|\hat u(\omega)|d\omega,
\end{equation}
where $\hat u$ is the Fourier transform of any extension of the
original function defined in $\Omega$ to the entire space $\mathbb
R^d$.  Here we need to point that $C_u$ might scale with dimension $d$.
The dependence on $d$ is improved by \cite{kurkova2002comparison, mhaskar2004tractability}. Especially, 
\cite{mhaskar2004tractability} improve this constant to be polynomial in $d$.

\subsection{An 1D example: a two point-boundary value problem}
As a proof of concept, let us discuss a very simple one dimensional
example. We focus on the following model problem:
\begin{equation}\label{eq:1d}
	\begin{split}
		-u''(x) = f,&\qquad x\in (0,1).\\
		\qquad u(0)=u(1) = 0.&
	\end{split}
\end{equation}

The exact solution $u\in H^1(0,1)$ satisfies that
\begin{equation}\label{eq:1d:exat}
	u = \mathop{\arg\min}_{v\in H_0^1(0,1)} E(v),
\end{equation}
where
$$
E(v) = \int_{0}^{1}(\frac{1}{2}|v(x)'|^2 -fv(x))dx.
$$

Given a grid
$$
\mathcal{T}_N : 0 = t_{0} <t_1< ... <t_{N+1} = 1.
$$
We define the space of $\relu$  DNNs with one hidden layer as follows:
$$
U = \{ u(x;t,\theta)|u(x;t,\theta) = \sum_{i=0}^{N} (\theta_{i+1}-\theta_i)\relu(x-t_i) \},
$$
where $\theta_i$ is the slope of piecewise linear function in
$[t_{i-1},t_i]$. In order to satisfy the condition $u(1;t,\theta)=0$, we have the constraint
$$
\theta_0 = 0,\quad \sum_{i=0}^{N+1}\theta_{i+1}(t_{i+1}-t_i)=0.
$$
We minimize the energy norm
$$
(t,\theta) = \mathop{\arg \min}_{t,\theta}   \int_{0}^{1}(\frac{1}{2}|u'(x;t,\theta)|^2 -fu(x;t,\theta))dx,\qquad u(x;t,\theta)\in U.
$$
where $t = (t_0,t_1,...,t_{N+1}),\theta =
(\theta_0,\theta_1,...,\theta_{N+1})$. We do the alternate iteration
as below,
\[
\begin{split}
\theta^{k+1} & =\mathop{\arg\min}_{\theta}\int_{0}^{1}\left(\frac{1}{2}|u'(x;t^k,\theta)|^2 -fu(x;t^k,\theta)\right)dx,\\
t^{k+1}      & =t^{k} - \eta  \nabla_{t} \left(\int_{0}^{1}\left(\frac{1}{2}|u'(x;t,\theta^{k+1})|^2 -fu(x;t,\theta^{k+1})\right)dx\right),
\end{split}
\]
where $\eta$ is the step-length.  Once $t$ is fixed, the minimization
problem is a quadratic optimization, which is the traditional finite element method. So we solve the FEM solution $u(x;t^k,\theta^{k+1})$ on grid $t$ and then compute the slope $\theta_i$ on each $[t_{i-1},t_i]$.

\begin{algorithm}
	\caption{Simulation $1D$ PDE}\label{algorithm:1dpde}  
	\begin{algorithmic}
		\STATE \textbf{Data: }{Grid $t$, Max iteration step $M$}.
		\STATE \textbf{Result: }{Optimal solution $u(x;t^*,\theta^*)$}. 
		\STATE Solve $\theta$ on the grid $t$;
		\WHILE {$k\leq M$}
		\STATE	$g = \nabla_{t} \left(\int_{0}^{1}\left(\frac{1}{2}|u'(x;t,\theta)|^2 -fu(x;t,\theta)\right)dx\right)$;
		\STATE		Find $\eta$ by line search;
		\STATE		$t\leftarrow t - \eta g $;
		\STATE		Solve $\theta$ on the grid $t$;
		\STATE		$k\leftarrow k+1$;
		\ENDWHILE
	\end{algorithmic}
\end{algorithm}
We choose the exact solution as
$$
u(x) = x(e^{-(x-\frac{1}{3})^2/K}-e^{-\frac{4}{9}/K}),
$$
with $K=0.01$.
In this numerical experiment,  the learning rate $\eta=0.5$, the max iteration
step is $200$, and the degrees of freedom $N=53$.

\begin{table}[htbp]
	\centering
	\begin{tabular}{ |c|c|c|c|c|c|c|  }
		\hline
		$N$ & $|u_{uFEM}-u|_1$ & $|u_{AFEM}-u|_{1}$& $|u_{DNN}-u|_{1}$ & $E(u_{uFEM})$ & $E(u_{AFEM})$ & $E(u_{DNN})$ \\
		\hline
		23 & 0.2779 &  0.1375 & 0.1094 & -0.7047 & -0.7338 &   -0.7373\\
		\hline
		37 & 0.1717 & 0.0760 & 0.0663 &  -0.7285 & -0.7404 & -0.7411\\
		\hline
		53 & 0.1193 & 0.0511 & 0.0456 & -0.7362 &  -0.7420 &-0.7422\\
		\hline
	\end{tabular}
	\caption{The $H^1$ semi-norm error and energy}
	\label{table:error}
\end{table}

\begin{figure}[htbp]
	\centering
	\begin{minipage}[t]{0.45\textwidth}
		\centering
		\includegraphics[width=6cm]{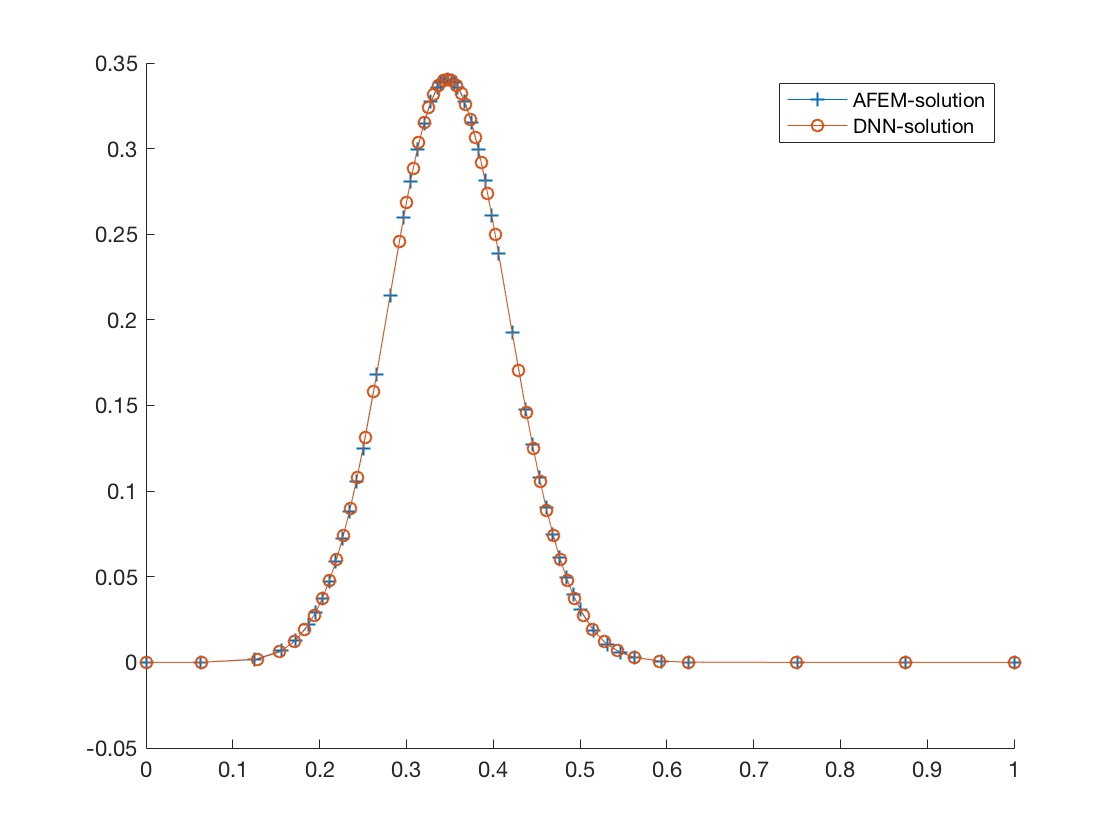}
	\end{minipage}
	\begin{minipage}[t]{0.45\textwidth}
		\centering
		\includegraphics[width=6cm]{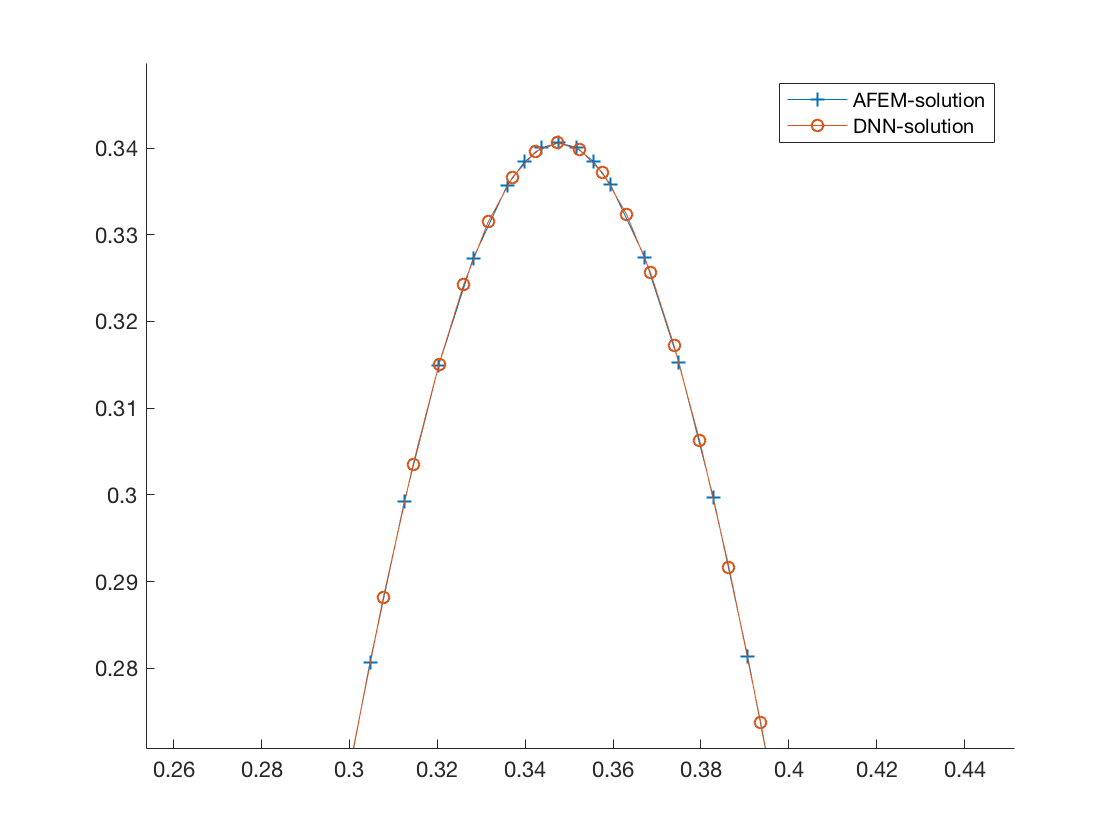}
	\end{minipage}
	\caption{The grid of AFEM and DNN(the right figure is the top of left).}
	\label{fig:1dsimu}
\end{figure}

At the beginning of the simulation, we use the
adaptive finite element method(AFEM) to get the adaptive grid from the
uniform grid. Next we construct DNN solution with the same degrees of
freedom. Then we minimize the energy and get the DNN solution. We
compare the energy and $H^1$ semi-norm error of uniform grid
solution(uFEM), AFEM solution and DNN solution. From Table~\ref{table:error}, the energy and $H^1$ semi-norm of the DNN
solution are smaller than $u_{AFEM}$ and $u_{uFEM}$, which implies that the DNNs can find better solution than AFEM. 
Figure~\ref{fig:1dsimu} shows the two different grid points on the
same graph, we can easily see the grid points are moving.

\section{Conclusion}\label{sec:conclusion}
By relating ReLU DNN models with linear finite element functions, we
provide some theoretical insights on why and how deep neural networks
work.  It is shown that ReLU DNN models with sufficiently many layers
(at least two) can reproduce all the linear finite element functions.
This in some sense provides some theoretical explanation of the
expressive power of deep learning models and also the necessity of
using deep layers in deep learning applications.

Two different approaches are discussed in this paper on the
representation of continuous piecewise linear functions by ReLU DNNs.
The first approach, as proposed in \cite{arora2016understanding} and
described in \S\ref{sec:cpwl}, leads to a DNN representation with a
relatively shallow network with $\lceil\log_2(d+1)\rceil$ hidden
layers but a relatively larger number of neurons.  The second
approach, presented in this paper and described in
\S\ref{sec:FEM_DNN}, leads to a representation that has a relatively
deeper network with $\lceil\log_2 k_h\rceil+1$ hidden layers (see
\eqref{kh}).  Further investigations are needed in the future to
combine these two approaches to obtain a more balanced representation.

The DNN representation of linear finite element functions opens a
door for theoretical explanation and possible improvement on the
application of the quantized weights in a convolution neural networks
(see \cite{yin2018binaryrelax}).  

One theoretically interesting question addressed in this paper
concerns the minimal number of layers that are needed in a DNN model
to reproduce general continuous piecewise linear functions.
Theorem~\ref{lowerbound} provides a partial answer to this question,
namely the minimal number of layers is at least $2$.  As a result, the
number of layers $\lceil\log_2(d+1)\rceil$ as given in
Theorem~\ref{main0} is optimal for $1\le d\le 3$.  It is still an open
question if this number is also optimal for $d\ge 4$.

This paper also briefly touches upon the application of DNN in
numerical solution of partial differential equations, which is a
topic that was investigated in the literature in 1990s and has
attracted much attention recently.  Our focus is on the comparison of
Galerkin type of discretization methods provided by adaptive linear
finite element methods and by deep neural networks.  When the
dimension $d$ is large, asymptotic approximation properties are
compared for these two different approaches in terms of the number of
the dimension $d$ and the number of the degrees of freedom.  When $d$ is
small, we use the simplest case $d=1$ to demonstrate that the deep
neural network would lead to a more accurate Galerkin approximation to
a differential equation solution than the adaptive finite element
method would under the assumption that the degrees of freedom are the
same in both cases but without comparing their computational costs.  This
preliminary study seems to indicate that deep neural network may
provide a potentially viable approach to the numerical solution of
partial differential equations for both high and low dimensions
although the underlying computational cost is a serious issue that may
or may not be properly addressed by further studies in the future.
\newpage

\begin{appendices}
	
\section{Lattice representation}
In this section, we will discuss the lattice representation of CPWL funtions in Theorem~\ref{lattice}. To begin with, let us recall Lemma~\ref{unique}, where we have $M$ subdomains $\{\tilde{\Omega}_k\}$ such that on each subdomain the arrangement of the $m$ local functions are fixed. We denote this kind of domain partition as unique-order region partition.

\begin{lemma}
	\label{1dcase}
	Let $p(t)$ be a continuous piecewise linear function and the unique-order region partition is
	$$0 = t_0 < t_1 < ... < t_{r+1} = 1$$
	assume the linear function on $[t_i,t_{i+1}]$ is $l_i(x) = k_i t+b_i$. If the parameters satisfy
	$$ b_0 > b_r,\qquad k_0 + b_0 > k_r + b_r $$
	which means
	$$l_0(t)> l_r(t)\quad \mbox{on}\quad [t_0,t_1]$$
	$$l_0(t)> l_r(t)\quad \mbox{on}\quad [t_r,t_{r+1}]$$
	Then there exists $l_p(t) = k_p t+ b_p$, such that
	$$b_p \ge b_0,\qquad k_p+b_p\le k_r+b_r$$  
	That is to say,
	$$l_0(t)\le l_p(t)\quad \mbox{on}\quad [t_0,t_1]$$
	$$l_p(t)\le l_r(t)\quad \mbox{on}\quad [t_r,t_{r+1}]$$
\end{lemma}
\begin{figure}[th]
	\centering
	\includegraphics[width=0.7\linewidth]{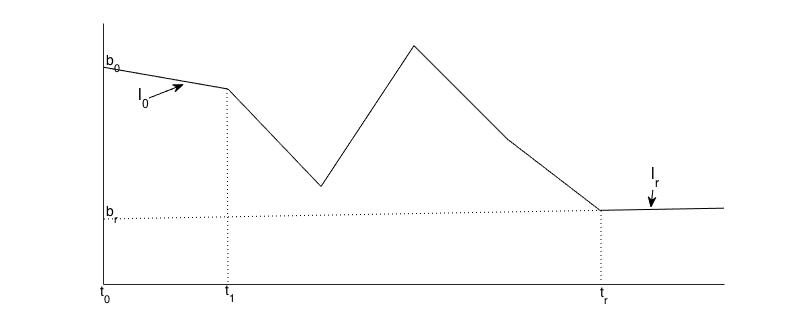}
\end{figure}

\begin{proof}
	
	Let $k_p=\min_i\{k_i\}$, $\Delta t_i=t_{i+1}-t_i$. 
	
	Since here we only involve linear functions, so we can represent each point $(t_i,y_i)$ by using $k_i$'s and $\Delta t_i$'s.
	
	Then the point $(t_p,b_0+\sum_{i=0}^{p-1}k_i\Delta t_i)$ is on $y=k_p t+b_p$, and we can write $b_p$ as following:
	$$b_p=b_0+\sum_{i=0}^{p-1}k_i\Delta t_i-k_p t_p=b_0+\sum_{i=0}^{p-1}(k_i-k_p)\Delta t_i$$
	Since here $k_p$ is the minimum, we have:
	$$b_p=b_0+\sum_{i=0}^{p-1}(k_i-k_p)\Delta t_i\ge b_0$$		
	And
	\begin{equation*}
	\begin{aligned}
	k_p+b_p&=k_p+b_0+\sum_{i=0}^{p-1}(k_i-k_p)\Delta t_i\\
	k_r+b_r&=k_r+b_0+\sum_{i=0}^{r-1}(k_i-k_p)\Delta t_i\\
	(k_r+b_r)-(k_p+b_p)&=k_r-k_p+\sum_{i=p}^{r-1}(k_i-k_p)\Delta t_i\ge 0
	\end{aligned}
	\end{equation*}
	which means we find the desired pair of $k$ and $p$.
	
	Notice here $k_p\ne k_0$ and $k_p\ne k_r$ by the assumptions in the lemma.
	
\end{proof}

\bigskip

The proof of Theorem~\ref{lattice} is as below.

\begin{proof}
	
	In each $\tilde \Omega_{k}$, consider the functions lie completely above $l_k$, and define the convex polynomial
	$$\Phi_{k}=\min_{i\in s_k} l_i$$
	here $s_k=\{i: l_i\ge l_k\ \mathrm{on}\ \tilde\Omega_{k}\}$.

	It's easy to see that in each $\tilde\Omega_{k}$, $\Phi_{k}=l_k=f$. Then define
	$$\Phi=\max_{k} \Phi_{k}$$
	and next we show that $\Phi_{k}(x)\le f(x)$ for all x and every $k$:

	For any fixed $k$, $\forall x\in\mathbb{R}^n$, if $x\in\tilde\Omega_{k}$, then $\Phi_{k}(x)=l_k(x)= f(x)$.
	
	If $x\notin\tilde\Omega_{k}$, then we suppose $x\in \tilde\Omega_{k'}$, so $f(x)=l_{k'}(x)$. Notice that here we have unique-order region, thus in each $\tilde\Omega_i$, the order of $l_k$ and $l_{k'}$ is fixed. 
	
	There're several situations:
	\begin{itemize}
		\item[(1)] $l_{k'}\ge l_{k}$ in $\tilde\Omega_{k}$, then $\Phi_{k}(x)\le l_{k'}(x)=f(x)$.
		\item[(2)] $l_{k'}< l_{k}$  in $\tilde\Omega_{k}$, then we consider the domain $\tilde\Omega_{k'}$:
		\begin{itemize}
			\item[(2a)]
			
			$l_{k'}\ge l_k$ in $\tilde\Omega_{k'}$. In this case we have:
			\begin{equation*}
			\Phi_{k}(x)\le l_k(x)\le l_{k'}(x)=f(x)
			\end{equation*}
			\item[(2b)]
			$l_{k'}<l_{k}$  in $\tilde\Omega_{k'}$. We take $x\in \tilde\Omega_{k}^\circ, x' \in \tilde\Omega_{k'}^\circ$. Then we have a path $L(\theta)$, the coordinate of the path is defined as $(x + \theta(x'-x),f(x + \theta(x'-x)))$ with $\theta\in[0,1]$(see Figure~\ref{parameterization}). It is just a piecewise linear function with the parameter $\theta$. Notice that the domain partition is now unique-order. So if we want to compare the order of the linear function,  we just compare one point value in that region. 
			\begin{figure}[th]
				\label{parameterization}
				\centering
				\includegraphics[width=0.7\linewidth]{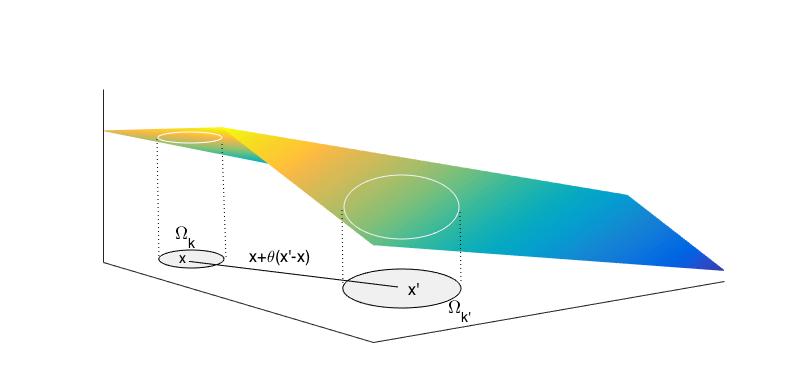}
				\caption{Parameterization of the path.}
			\end{figure}

			Then by Lemma \ref{1dcase}, there must exist $l_t$ with $t\ne k,k'$ and
			\begin{equation*}
			\begin{aligned}
			&l_t\le l_{k'}\quad \mathrm{on}\ \tilde\Omega_{k'}\\
			&l_t\ge l_{k}\quad \mathrm{on}\ \tilde\Omega_{k}
			\end{aligned}
			\end{equation*}
			Then we should have:
			$$\Phi_{k}(x)\le l_t(x)\le l_{k'}(x)=f(x)$$

			Thus for every $\Phi_{k}$, we have $\Phi_{k}(x)\le f(x)$ for all $x$. So if we take the maximum over all these functions, we should have:
			$$f(x)=\max_{k}\min_{i\in s_k}\{l_i\}$$
			This is exactly the desired form. Here $|s_k|\le m$, and the number of $\Phi_{k}$ depends on the domain partition we do. 
		\end{itemize}
	\end{itemize}
\end{proof}

The drawback of this representation is that the number of $M$ may be too large, so we want to deal with other domain partitions, for example, partitions that produce a set of convex regions. The following theorem is an improvement of Theorem ~\ref{lattice}.

\bigskip
\begin{theorem}
	Let $f:\mathbb{R}^d\to\mathbb{R}$ be a continuous piecewise linear function with finite pieces defined by the distinct local linear functions $l_i$, $1\le i\le m$. Let $\{C_k\}$ be a set of $S$ convex regions of $f$ that determine a domain partition, so that the function $f$ is described in each region $C_k$ by the same linear function $l_k$. Then there exist finite non-empty subsets of $\{1,2,\dots,m\}$, say $s_k$, $1\le k\le S$, such that 
	\begin{equation}
	f(x)=\max_{1\le k\le S}\{\min_{i\in s_k} l_i\}
	\end{equation}
\end{theorem}

\begin{proof}
	In each $C_k$, define:
	$$\Psi_{k}=\min_{i\in s_k} l_i$$
	here $s_k=\{i: l_i\ge l_k\ \mathrm{on}\ C_{k}\}$.
	Then define:
	$$\Psi=\max_{1\le k\le S}\Psi_k$$
	We should have $f=\Psi$ for all $x$.

	If the arrangement inside a convex region $C_k$ is the same, then the proof should be the same as Theorem ~\ref{lattice}. If not,  the contribution of each region $C_k$ can be considered as the union of the contributions of its unique-order subsets: $\Psi_k=\max_{1\le i\le M_k}\Phi_{i}$, here $\Phi_i$ is defined as in Theorem \ref{lattice}. We can see that $\sum_{k=1}^{S}M_k$ is no less than $M$, the number of unique-order regions. By applying properties of lattices, we have:
	$$\Psi=\max_{1\le k\le S}\Psi_k=\max_{1\le i\le M}\Phi_i$$
	so according to Theorem \ref{lattice}, $\Psi=f$ for all $x$ in the domain.
\end{proof}

	\section{Proof of Lemmas}
	
	In this section, we will show the proofs of the lemmas used in previous sections.
	
	\subsection{Proof of Lemma~\ref{bar-alpha}}
	\begin{proof}
		We have the identity
		\begin{equation}\label{trans}
		\begin{aligned}
		\max\{f,g,\alpha g+h\}&=\max\{f,\max\{g,\alpha g+h\}\}\\
		&=\max\{f,\max\{g-\bar{\alpha}h,\alpha g+h-\bar{\alpha}h\}+\bar{\alpha}h\}\\
		&=\max\{f,\max\{g-\bar{\alpha}h,\alpha (g-\bar{\alpha}h)\}+\bar{\alpha}h\}
		\end{aligned}
		\end{equation}
		
		When $\alpha>1$, if $g-\bar{\alpha}h\ge0$, then (\ref{trans}) becomes
		$$
		\max\{f,\max\{1,\alpha\}\max\{g-\bar{\alpha}h,0\}+\bar{\alpha}h\}
		$$
		and the right hand side (RHS) of (\ref{fir}) becomes the same.
		
		If $g-\bar{\alpha}h\le0$, then (\ref{trans}) becomes 
		$$
		\max\{f,\min\{1,\alpha\}\min\{g-\bar{\alpha}h,0\}+\bar{\alpha}h\}
		$$
		and the RHS of (\ref{fir}) is also this expression.
		
		For $\alpha<1$, similarly, the identity (\ref{fir}) always holds. Further, if consider the cases $\alpha>1$, $0<\alpha<1$ and $\alpha<0$ respectively, we have the following important identity:
		
		\begin{equation}
		\max\{f,g,\alpha g+h\}=	\sigma_1\max\{f,g,\bar{\alpha}h\}+\sigma_2\max\{f,\alpha g+h,\bar{\alpha}h\}+\sigma_3\max\{f,\bar{g}\}
		\end{equation}
		here $\sigma_k\in\{1,-1\}$ and $\bar{g}$ is one of $g$, $\alpha g+h$ or $\bar{\alpha}h$.
	\end{proof}
	
	\subsection{Proof of Lemma~\ref{reduce}}
	\begin{proof}
		Let $l_i(x)=a_{i0}+\bar{a}_i^T x$ with $a_{i0}\in\mathbb{R}$ and $\bar{a}_i\in\mathbb{R}^n$ for $1\le i\le L$. Assume there are at most $\bar{n}$ linearly independent $\bar{a}_i$, $\bar{n}\le n$. Without loss of generality, assume $\bar{a}_1,\dots,\bar{a}_{\bar{n}}$ are linearly independent. Then by basic linear algebra, we know that 
		$$
		l_L=\sum_{j=1}^{\bar{n}}\alpha_j l_j+\alpha_0.
		$$
		Denote
		$$
		\mu(x)=\max\{l_{\bar{n}+1},\dots,l_{L-1}\}.
		$$
		Then
		\begin{equation}\label{1}
		\max\{c_0,l_1,\dots,l_L\}=\max\{c_0,l_1,\dots,l_{\bar{n}},\mu(x),l_L\}.
		\end{equation}
		
		If $\alpha_j = 0$ for each $1\le j\le \bar{n}$, then by taking $\max\{c_0,\alpha_0\}$, we already make the RHS of (\ref{1}) as the RHS of (\ref{goal}). Otherwise, $\exists \alpha_j\ne0$ for some $1\le j\le \bar{n}$, we can assume that $\alpha_{\eta} \neq 0, \alpha_{\eta+1}= ... = \alpha_{\bar{n}}$, so $$l_L=\sum_{j=1}^{\eta}\alpha_j l_j+\alpha_0,\qquad \eta\le \bar{n}.$$
		
		If $\alpha_i=1$ for each $1\le i\le\eta$, we can do the following linear transformation $(l'_1,\dots,l'_{\bar{n}})^T=A(l_1,\dots,l_{\bar{n}})^T$, where
		\begin{equation*}
		\begin{aligned}
		l'_i&=l_i,\qquad &\mathrm{for}\ i\ne\eta,\\
		l'_i&=\sum_{j=1}^{\eta} l_j+\alpha_0,\qquad &\mathrm{for}\ i=\eta.\\
		\end{aligned}
		\end{equation*}
		
		Then (\ref{1}) equals 
		$$
		\max\{c_0,l'_1,\dots,l'_{\bar{n}},\mu(x),-\alpha_0-l'_1-\dots-l'_{\eta-1}+l'_{\eta}\}.
		$$
		in this case, the coefficients of $l'_1,\dots,l'_{\eta-1}$ are not 1.
		
		So we can assume there is at least one $\alpha_i\ne1$ for $1\le i\le\eta$, say $\alpha_{\eta}\ne1$. Let
		\begin{equation*}
		\begin{aligned}
		f&=\max\{c_0,\mu(x),l_1,\dots,l_{\eta-1},l_{\eta+1},\dots,l_{\bar{n}}\},\\
		g&=l_{\eta},\\
		h&=\sum_{j=1}^{\eta-1}\alpha_j l_j+\alpha_0.
		\end{aligned}
		\end{equation*}
		by (\ref{key-reduce}) we have:
		\begin{eqnarray}\label{2}
		\max\{c_0,l_1,\dots,l_{\bar{n}},\mu(x),l_L\}&&=\max\{f,g,\alpha_{\eta} g+h  \}\\ 
		\label{3}
		&&=\sigma_1\max\{f,g,\frac{\sum_{j=1}^{\eta-1}\alpha_j l_j+\alpha_0}{1-\alpha_{\eta}}  \}\\ 
		\label{4}
		&&+\sigma_2\max\{f,\alpha_{\eta} g+h,\frac{\sum_{j=1}^{\eta-1}\alpha_j l_j+\alpha_0}{1-\alpha_{\eta}}   \}\\ 
		\label{5}
		&&+\sigma_3\max\{f,\bar{g} \}.
		\end{eqnarray}
		(\ref{5}) is already the desired form, because now we only take maximum  over $L-1$ linear functions and one constant.
		
		As for the (\ref{3}), notice that now we have eliminated $l_{\eta}$ in the third expression. So continue this procedure, at last we will only have constant in the last expression, by taking maximum of this constant and $c_0$, we can reduce one term in the max expression.
		
		For (\ref{4}), consider the linear transformation $(l''_1,\dots,l''_{\bar{n}})^T=B(l_1,\dots,l_{\bar{n}})^T$:
		\begin{equation*}
		\begin{aligned}
		l''_i&=l_i,\qquad &\mathrm{for}\ i\ne\eta,\\
		l''_i&=\sum_{j=1}^{\eta}\alpha_{j} l_j+\alpha_0,\qquad &\mathrm{for}\ i=\eta.\\
		\end{aligned}
		\end{equation*}
		So (\ref{4}) becomes
		$$
		\max\{c_0,\mu,l''_1,\dots,l''_{\bar{n}},\sum_{j=1}^{\eta-1}\alpha_j l''_j+\alpha_0\}.
		$$
		Then it is the same as (\ref{3}). Follow the same steps as for \eqref{3}, we can achieve the desired result.
	\end{proof}
	
	\begin{figure}[ht]
		\label{Redu}
		\begin{center}
			\begin{tikzpicture}[>=triangle 45,font=\sffamily]
			\node (X)  {(\ref{2})};
			\node (Y) [below left=1cm and 2cm of X]  { (\ref{3})};
			\node (Z) [below right=1cm and 2cm of X] {(\ref{4})};
			\node (O) [below =1 cm of X] {(\ref{5})};
			\node (U) [below left=1cm and 1.2cm of Y] { ...};
			\node (V) [below right=1cm and 1.2cm of Y] {... };
			\node (W) [below =1cm of Y] { };
			\node (U1)[below left=1cm and 1.2cm of Z] {...};
			\node (V1)[below right=1cm and 1.2cm of Z] {...};
			\node (W1) [below =1cm of Z] {};
			\node (T1)[below left=1cm and 0.5cm of U] { };
			\node (T2)[below right=1cm and 0.5cm of U] { };
			\node (T3)[ below =1cm of U] { };
			\node (T4)[ below left=1cm and 0.5cm of V] { };
			\node (T5)[ below right=1cm and 0.5cm of V] { };
			\node (T6)[ below =1cm of V] { };
			\node (T7)[below right=1cm and 0.5cm of U1] { };
			\node (T8)[ below left=1cm and 0.5cm of U1] { };
			\node (T9)[below =1cm of U1] { };
			\node (T10)[below right=1cm and 0.5cm of V1] { };
			\node (T11)[ below left=1cm and 0.5cm of V1] { };
			\node (T12)[below =1cm of V1] { };
			\draw [semithick,->] (X) -- (Y);
			\draw [semithick,->] (X) -- (Z);
			\draw [semithick,->] (X) -- (O);
			\draw [semithick,->] (Y) -- (U);
			\draw [semithick,->] (Y) -- (V);
			\draw [semithick,->] (Y) -- (W);
			\draw [semithick,->] (Z) -- (U1);
			\draw [semithick,->] (Z) -- (V1);
			\draw [semithick,->] (Z) -- (W1);
			\draw [semithick,->] (U) -- (T1);
			\draw [semithick,->] (U) -- (T2);
			\draw [semithick,->] (U) -- (T3);
			\draw [semithick,->] (V) -- (T4);
			\draw [semithick,->] (V) -- (T5);
			\draw [semithick,->] (V) -- (T6);
			\draw [semithick,->] (U1) -- (T7);
			\draw [semithick,->] (U1) -- (T8);
			\draw [semithick,->] (U1) -- (T9);
			\draw [semithick,->] (V1) -- (T10);
			\draw [semithick,->] (V1) -- (T11);
			\draw [semithick,->] (V1) -- (T12);
			\end{tikzpicture}
			\caption{The process of reducing one term.}
		\end{center}
	\end{figure}
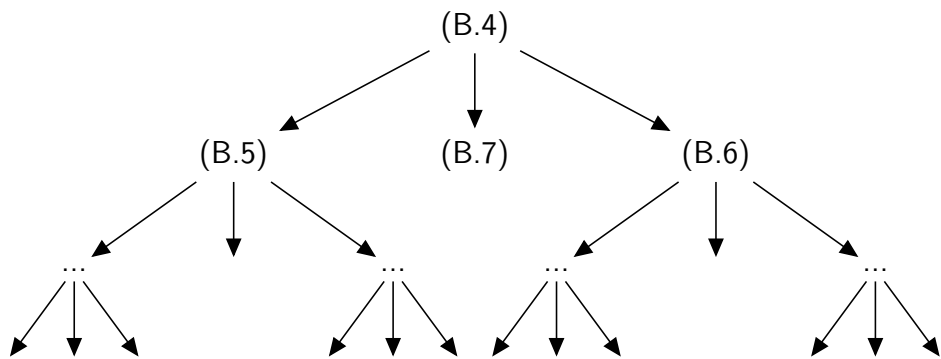
	\begin{remark}
		\label{timesestimate}
		Whenever we eliminate one $l_i$ in the expression of $l_L$, we will gain 3 terms, which is (\ref{3}-\ref{5}). Among these three terms, (\ref{5}) is in desired form, and we need to continue to use (\ref{key-reduce}) for (\ref{3}) and (\ref{4}) until we only have constant. Note that in the proof, $\eta\le n$.  By this procedure, we will gain at most $2^{n+1}-1$ terms (see Figure~\ref{Redu}). 
	\end{remark}
	
	\subsection{Proof of Lemma~\ref{size} }
	\begin{proof}
		We prove the lemma by induction. When $m=1$, the case is trivial. When $m\ge2$, consider $g_1:=\max\{f_1,\dots,f_{\lfloor \frac{m}{2}\rfloor}\}$ and $g_2:=\max\{f_{\lceil \frac{m}{2}\rceil},\dots,f_m\}$. By indcution hypothesis, $g_1$ and $g_2$ can be represented by ReLU DNNs of depths at most $\max\{k_1,\dots,k_{\lfloor \frac{m}{2}\rfloor}\}+\lceil\log({\lfloor \frac{m}{2}\rfloor})\rceil+1$ and $\max\{k_{\lceil \frac{m}{2}\rceil},\dots,k_m\}+\lceil\log({\lceil \frac{m}{2}\rceil})\rceil+1$ respectively, and sizes at most $s_1+\dots+s_{\lfloor\frac{m}{2}\rfloor}+4(2\lfloor\frac{m}{2}\rfloor-1)$ and $s_{\lceil\frac{m}{2}\rceil}+4(2\lceil\frac{m}{2}\rceil-1)$ respectively. Then consider the function $F:\mathbb{R}^n\to\mathbb{R}^2$ defined as $F(x)=(g_1(x),g_2(x))$, this can be represented by a ReLU DNN with depth at most $\max\{k_1,\dots,k_m\}+\lceil\log({\lceil \frac{m}{2}\rceil})\rceil+1$ and size at most $s_1+\dots+s_m+4(2m-2)$.
		
		Since $\max\{x,y\}$ can be represented by a 2-layer ReLU DNN with size 4, we know that $f$ can be represented by a ReLU DNN of depth at most $\max\{k_1,\dots,k_m\}+\lceil\log(m)\rceil+1$  and size at most $s_1+\dots+s_m+4(2m-1)$.
	\end{proof}
\end{appendices}


\begin{thebibliography}{10}
	\bibitem{lecun2015deep}Y.~LeCun, Y.~Bengio and G.~Hinton, Deep learning,
	\emph{nature}, \textbf{521}(2015), 436.
	
	\bibitem{cohen2016expressive} N.~Cohen, O.~Sharir and A.~Shashua, On the expressive power of deep learning: A tensor analysis, in \emph{Conference on Learning Theory}, 2016, 698-728. 
	
	\bibitem{hornik1989multilayer}K.~Hornik, M.~Stinchcombe and H.~White, Multilayer feedforward networks are universal approximators,
	\emph{Neural networks}, \textbf{2}(1989), 359-366.
	
	\bibitem{cybenko1989approximation}G.~Cybenko, Approximation by superpositions of a sigmoidal function,
	\emph{Mathematics of control, signals and systems}, \textbf{2}(1989), 303-314.
	
	\bibitem{jones1992simple} L.K.~Jones, A simple lemma on greedy approximation in Hilbert space and convergence rates for projection pursuit regression and neural network training,
	\emph{The annals of Statistics}, \textbf{20}(1992), 608-613.
	
	\bibitem{barron1993universal}A.R.~Barron, Universal approximation bounds for superpositions of a sigmoidal function,
	\emph{IEEE Transactions on Information theory}, \textbf{39}(1993), 930-945.
	
	\bibitem{ellacott1994aspects}S.W.~Ellacott, Aspects of the numerical analysis of neural networks,
	\emph{Acta numerica}, \textbf{3}(1994), 145-202.
	
	\bibitem{pinkus1999approximation}A.~Pinkus, Approximation theory of the MLP model in neural networks,
	\emph{Acta numerica}, \textbf{8}(1999), 143-195.
	
	\bibitem{leshno1993multilayer}M.~Leshno, V.Y.~Lin, A.~Pinkus and S.~Schocken, Multilayer feedforward networks with a nonpolynomial activation function can approximate any function,
	\emph{Neural networks}, \textbf{6}(1993), 861-867.
	
	\bibitem{nair2010rectified} V.~Nair,and G.~Hinton, Rectified linear units improve restricted boltzmann machines, in \emph{Proceedings of the 27th international conference on machine learning (ICML-10)}, 2010, 807-814.
	
	\bibitem{shaham2016provable}U.~Shaham, A.~Cloninger and R.R.~Coifman, Provable approximation properties for deep neural networks,
	\emph{Applied and Computational Harmonic Analysis}, 2016.
	
	\bibitem{klusowski2018approximation} J.M.~Klusowski and A.R.~Barron, Approximation by Combinations of ReLU and Squared ReLU Ridge Functions with $l^1$ and $l^0$ Controls, 2018.
	
	\bibitem{zhou2018universality} D.X.~Zhou, Universality of Deep Convolutional Neural Networks,
	\emph{arXiv preprint arXiv:1805.10769}, 2018.
	
	\bibitem{arora2016understanding} R.~Arora, A.~Basu, P.~Mianjy and A.~Mukherjee, Understanding deep neural networks with rectified linear units,  \emph{arXiv preprint arXiv:1611.01491}, 2016.
	
	\bibitem{tarela1999region}J.M.~Tarela and M.V.~Martinez, Region configurations for realizability of lattice piecewise-linear models,
	\emph{Mathematical and Computer Modelling}, \textbf{30}(1999), 17-27.
	
	\bibitem{ciarlet2002finite} P.G.~Ciarlet, The finite element method for elliptic problems,
	\emph{Classics in applied mathematics}, \textbf{40}(2002), 1-511.
	
	\bibitem{yin2018binaryrelax}P.~Yin, S.~Zhang, J.~Lyu, S.~Osher, Y.~Qi and J.~Xin, BinaryRelax: A Relaxation Approach For Training Deep Neural Networks With Quantized Weights,
	\emph{arXiv preprint arXiv:1801.06313}, 2018.
	
	\bibitem{lecun1998gradient}Y.~LeCun, L.~Bottou, Y.~Bengio and P.~Haffner, Gradient-based learning applied to document recognition,
	\emph{Proceedings of the IEEE}, \textbf{86}(1998), 2278-2324.
	
	\bibitem{krizhevsky2009learning}A.~Krizhevsky and G.~Hinton, Learning multiple layers of features from tiny images, Technical report, University of Toronto, 2009.
	
	\bibitem{li2016ternary}F.~Li, B.~Zhang and B.~Lin, Ternary weight networks,
	\emph{arXiv preprint arXiv:1605.04711}, 2016.
	
	\bibitem{gobovic1994analog} D.~Gobovic and M.E.~Zaghloul, Analog cellular neural network with application to partial differential equations with variable mesh-size, in \emph{Circuits and Systems, 1994. ISCAS'94., 1994 IEEE International Symposium on}, IEEE, 1994, 359-362.
	
	\bibitem{meade1994numerical} A.J.~Meade Jr and A.A.~Fernandez, The numerical solution of linear ordinary differential equations by feedforward neural networks,
	\emph{Mathematical and Computer Modelling}, \textbf{19}(1994), 1-25.
	
	\bibitem{meade1994solution} A.J.~Meade Jr and A.A.~Fernandez, Solution of nonlinear ordinary differential equations by feedforward neural networks,
	\emph{Mathematical and Computer Modelling}, \textbf{20}(1994), 19-44.
	
	\bibitem{weinan2017deep} W.~E, J.~Han and A.~Jentzen, Deep learning-based numerical methods for high-dimensional parabolic partial differential equations and backward stochastic differential equations,
	\emph{Communications in Mathematics and Statistics}, \textbf{5}(2017), 349-380.
	
	\bibitem{han2017overcoming} J.~Han, A.~Jentzen, and W.~E, Overcoming the curse of dimensionality: Solving high-dimensional partial differential equations using deep learning,
	\emph{arXiv preprint arXiv:1707.02568}, 2017.
	
	\bibitem{khoo2017solving} Y.~Khoo, J.~Lu and L.~Ying, Solving parametric PDE problems with artificial neural networks,
	\emph{arXiv preprint arXiv:1707.03351}, 2017.
	
	\bibitem{devore1998nonlinear} R.A.~DeVore, Nonlinear approximation,
	\emph{Acta numerica}, \textbf{7}(1998), 51-150.
	
	\bibitem{brenner2007mathematical} S.~Brenner and R.~Scott, \emph{The mathematical theory of finite element methods(Vol.~15)}, Springer Science \& Business Media, 2007. 
	
	\bibitem{nochetto2011primer} R.H.~Nochetto and A.~Veeser, Primer of adaptive finite element methods, in \emph{Multiscale and adaptivity: modeling, numerics and applications}, Springer, Berlin, Heidelberg, 2011, 125-225. 
	
	\bibitem{wang2005generalization}S.~Wang and X.~Sun, Generalization of hinging hyperplanes,
	\emph{IEEE Transactions on Information Theory}, \textbf{51}(2005), 4425-4431.
	
	\bibitem{dai2016r} J.~Dai, Y.~Li, K.~He and J.~Sun, R-fcn: Object detection via region-based fully convolutional networks, in \emph{Advances in neural information processing systems}, Curran Associates, Inc., 2016, 379-387.
	
	\bibitem{lagaris1998artificial} I.E.~Lagaris, A.~Likas and D.I.~Fotiadis, Artificial neural networks for solving ordinary and partial differential equations,
	\emph{IEEE Transactions on Neural Networks}, \textbf{9}(1998), 987-1000.
	
	\bibitem{e2018deep}W.~E and B.~Yu, The Deep Ritz Method: A Deep Learning-Based Numerical Algorithm for Solving Variational Problems,
	\emph{Communications in Mathematics and Statistics}, \textbf{6}(2018), 1-12.
	
	\bibitem{li2001moving} R.~Li, T.~Tang and P.~Zhang, Moving mesh methods in multiple dimensions based on harmonic maps,
	\emph{Journal of Computational Physics}, \textbf{170}(2001), 562-588.
	
	\bibitem{li2002moving} R.~Li, T.~Tang and P.~Zhang, A moving mesh finite element algorithm for singular problems in two and three space dimensions,
	\emph{Journal of Computational Physics}, \textbf{177}(2002), 365-393.
	
	\bibitem{xu2017deep}J.~Xu, Deep Neural Networks and Multigrid Methods({L}ecture {N}otes), Penn State University, 2017.
	
	\bibitem{idelsohn2003meshless} S.R.~Idelsohn, E.~Onate, N.~Calvo and F.D.~Pin, The meshless finite element method,
	\emph{International Journal for Numerical Methods in Engineering}, \textbf{58}(2003), 893-912.
	
	\bibitem{liu2002mesh} G.R.~Liu, \emph{Mesh free methods: moving beyond the finite element method},CRC press, 2002. 
	
	\bibitem{yagawa1996free} G.~Yagawa and T.~Yamada, Free mesh method: a new meshless finite element method,
	\emph{Computational Mechanics}, \textbf{18}(1996), 383-386.
	
	\bibitem{kurkova2002comparison} V.~Kurkov{\'a} and M.~Sanguineti, Comparison of worst case errors in linear and neural network approximation,
	\emph{IEEE Transactions on Information Theory}, \textbf{48}(2002), 264-275.
	
	\bibitem{mhaskar2004tractability} H.N.~Mhaskar, On the tractability of multivariate integration and approximation by neural networks,
	\emph{Journal of Complexity}, \textbf{20}(2004), 561-590.
	
	
	
	
	
	
	
\end{thebibliography}
\end{document}